\documentclass[12pt,a4paper,twoside,reqno]{amsart}
\title[Holomorphic string algebroids]{Holomorphic string algebroids}
\author[M. Garcia-Fernandez, R. Rubio, and C. Tipler]{Mario Garcia-Fernandez, Roberto Rubio, and Carl Tipler}

\address{Dep. Matem\'aticas, Universidad Aut\'onoma de Madrid, and Instituto de Ciencias Matem\'aticas (CSIC-UAM-UC3M-UCM), Cantoblanco, 28049 Madrid, Spain}
\email{mario.garcia@icmat.es}

\address{Weizmann Institute of Science, 234 Herzl St, Rehovot 76100, Israel} 
\curraddr{Universitat Aut\`onoma de Barcelona, 08193 Barcelona, Spain}
\email{roberto.rubio@uab.es}

\address{LMBA, UMR CNRS 6205; D\'epartement de
  Math\'ematiques, Universit\'e de Bretagne Occidentale, 6, avenue
  Victor Le Gorgeu, 29238 Brest Cedex 3 France}
\email{carl.tipler@univ-brest.fr}

\thanks{\noindent This project has received funding from the European Union's Horizon 2020 research and innovation programme under the Marie Sklodowska-Curie grant agreement No 655162. MGF is supported by Universidad Aut\'onoma de Madrid and was initially supported by a Marie Sklodowska-Curie grant. RR is supported by the Weizmann Institute (ERC StG grant 637912, ISF grant 687/13) and was initially supported by IMPA. CT is supported by the French government ``Investissements d'Avenir'' program ANR--11--LABX--0020--01 and ANR project EMARKS No ANR--14--CE25--0010.}

\usepackage{hyperref,url}
\usepackage{amssymb,latexsym}
\usepackage{amsmath}
\usepackage[latin1]{inputenc}
\usepackage{enumerate}
\usepackage[all]{xy} \CompileMatrices
\usepackage{amscd}

\usepackage{tikz}
\usetikzlibrary{matrix}

\usepackage{slashed} 

\SelectTips{cm}{12} 
\newtheorem{theorem}{Theorem}[section]
\newtheorem{lemma}[theorem]{Lemma}
\newtheorem{corollary}[theorem]{Corollary}
\newtheorem{proposition}[theorem]{Proposition}

\newtheorem{definition}[theorem]{Definition}
\newtheorem{definition-theorem}[theorem]{Definition-Theorem}
\newtheorem{example}[theorem]{Example}
\newtheorem{remark}[theorem]{Remark}

\numberwithin{equation}{section} 

\newcommand{\tr}{\operatorname{tr}}
\newcommand{\Id}{\operatorname{Id}}

\newcommand{\Ker}{\operatorname{Ker}}
\newcommand{\ad}{\operatorname{ad}}

\newcommand{\Aut}{\operatorname{Aut}}
\newcommand{\dbar}{\bar{\partial}}

\newcommand{\CC}{{\mathbb C}}

\newcommand{\RR}{{\mathbb R}}
\newcommand{\ZZ}{{\mathbb Z}}

\newcommand{\surj}{\to\kern-1.8ex\to}

\newcommand{\lra}[1]{\stackrel{#1}{\longrightarrow}}

\newcommand{\xra}[1]{\xrightarrow{#1}}

\newcommand{\cA}{\mathcal{A}}
\newcommand{\cC}{\mathcal{C}}

\newcommand{\cM}{\mathcal{M}}
\newcommand{\cP}{\mathcal{P}}

\newcommand{\cG}{\mathcal{G}}
\newcommand{\cL}{\mathcal{L}}
\newcommand{\cO}{\mathcal{O}}

\newcommand{\cS}{\mathcal{S}}

\newcommand{\Lie}{\operatorname{Lie}}

\newcommand{\cH}{\mathcal{H}} 

\def\Om{\Omega}

\def\Lie{\mathrm{Lie}}

\def\Im{\mathrm{Im}}
\def\Id{\mathrm{Id}}

\def\cA{\mathcal{A}}
\def\cB{\mathcal{B}}

\def\cG{\mathcal{G}}
\def\cH{\mathcal{H}}
\def\cP{\mathcal{P}}

\def\fL{\mathfrak{L}}

\def\faut{\mathfrak{Aut}}

\def\del{\partial}
\def\delb{\overline\partial}

\newcommand{\st}{\;|\;}

\newcommand{\C}{{\mathbb{C}}}

\newcommand{\fg}{\mathfrak{g}}

\newcommand{\la}{\langle}
\newcommand{\ra}{\rangle}

\newcommand{\Spin}{\mathrm{Spin}}

\renewcommand{\Im}{\mathrm{Im}}

\begin{document}

\maketitle

\begin{abstract}
We introduce the category of \emph{holomorphic string algebroids}, whose objects are Courant extensions of Atiyah Lie algebroids of holomorphic principal bundles, and whose morphisms correspond to \emph{inner morphisms} of the underlying holomorphic Courant algebroids. This category provides natural candidates for Atiyah Lie algebroids of holomorphic principal bundles for the (complexified) \emph{string group} and their morphisms. Our main results are a classification of string algebroids in terms of \v Cech cohomology, and the construction of a locally complete family of deformations of string algebroids via a differential graded Lie algebra.
\end{abstract}


\section{Introduction}
\label{sec:intro}

The first Pontryagin class was interpreted in \cite{Bressler} as the obstruction to define certain sheaves of chiral differential operators \cite{BD,GMS,MSV}, associated to a Lie algebroid $A\xra{\pi} TX$ endowed with an invariant pairing on the kernel of the anchor map $\pi$. 
It turns out that the classification problem for these 
sheaves 
reduces to the study of Courant extensions of $A$, given by a Courant algebroid $Q$ and a short exact sequence
\begin{equation}\label{eq:sesintro}
	\xymatrix{
		0 \ar[r] & T^*X \ar[r] & Q \ar[r]^\rho & A \ar[r] & 0
	}
\end{equation}
such that $\rho$ is a bracket-preserving map. A complete classification of extensions of the form \eqref{eq:sesintro} was obtained in 
\cite{ChStXu} in a differential-geometric setting.

In this work we introduce the category of \emph{holomorphic string algebroids} for a complex manifold $X$ and a complex Lie group $G$. 
Objects in this category are Courant extensions $(Q,P,\rho)$ of Atiyah Lie algebroids $A_P$ of holomorphic principal $G$-bundles $P$ over $X$, while morphisms correspond to \emph{inner morphisms} of $Q$ in the sense of \v Severa \cite{Severa}. Upon fixing the principal bundle $P$, these special morphisms define a subcategory of the Courant extensions of $A_P$ considered in \cite{Bressler}. When $G$ is trivial, a string algebroid is simply an exact holomorphic Courant algebroid 
\cite{G2}. Our main focus is the study of two basic pieces of the theory of holomorphic string algebroids: their classification and their deformation theory, which will be described below.

%
%

A geometric interpretation of the notion of holomorphic string algebroid can be found in higher gauge theory. Recent work in \cite{Xu} constructs a morphism from the stack of smooth principal $2$-bundles of string groups to the stack of smooth transitive Courant algebroids (with connection data), which can be understood as the process of associating the ``higher Atiyah algebroid'' to the string principal bundle. This morphism turns out to be neither injective nor surjective, mainly due to the fact that (leaving integrality conditions aside) the gluing data for objects in the image, rather than given by general Courant algebroid automorphisms, corresponds to the more restrictive inner automorphisms of \v Severa. Thus, our holomorphic string algebroids provide natural candidates for Atiyah Lie algebroids of holomorphic principal bundles for the (complexified) string group. 

A initial motivation for this paper was to find an analogue 
in the holomorphic category of a general picture found in previous work by the authors (see \cite{grt}, and \cite{cgt} jointly with A. Clarke) 
about the moduli problem for the Hull-Strominger system of partial differential equations \cite{HullTurin,Strom}. In the smooth setup,
these equations are defined on a compact spin manifold $M$ endowed with a principal bundle $P$ with compact structure group $K$ and vanishing first Pontryagin class. \emph{Gauge symmetries} are given by the group of equivariant diffeomorphisms of $P$ that project to the identity component of $\operatorname{Diff}(M)$. The corresponding moduli space $\cM$ has a natural map
\begin{equation}\label{eq:vartheta}
\vartheta \colon \cM \to H^3_{str}(P,\RR)/I_{[P]},
\end{equation}
where $H^3_{str}(P,\RR)$ denotes the $H^3(M,\RR)$-torsor of real string classes of $P$ \cite{Redden}, and $I_{[P]} \subset H^3(M,\RR)$ is an additive subgroup associated to the gauge group of $P$ (see Appendix \ref{sec:smooth}). Remarkably, a level set for $\vartheta$ can be regarded as a moduli space of solutions of the \emph{Killing spinor equations} on a smooth string algebroid, as defined in \cite{GF3}, modulo the action of \v Severa's automorphisms. When $K = \operatorname{Spin}(r)$, the integral elements of $H^3_{str}(P,\RR)$ correspond to isotopy classes of lifts of $P \colon M \to  B\operatorname{Spin}(r)$ to the classifying space of the string~group~\cite{Waldorf}. To some extent, a holomorphic version of this picture is provided by our Proposition \ref{theo:deRhamC}.


In a sequel to the present work \cite{grt2}, in collaboration with C. Shahbazi, we use holomorphic string algebroids to take some first steps towards an analogue of the Donaldson-Uhlenbeck-Yau Theorem in higher gauge theory for complex non-K\"ahler manifolds. Key to our development in \cite{grt2} is the classification theorem of string algebroids, Theorem \ref{th:classification} in the present work, which identifies the `holomorphic data' underlying a solution of the Hull-Strominger system. In a nutshell, a tuple $(Q,P,\rho)$ is to a solution of the Hull--Strominger system the same as a holomorphic bundle is to a solution of the Hermite-Yang-Mills equations.
Moreover, we expect that our deformation theory for string algebroids, Theorem \ref{thm:slice}, 
will play an important role in future studies about the existence and uniqueness problem, as well as in the moduli problem for the Hull-Strominger system.

Holomorphic string algebroids seem to appear in recent physical advances in the understading of the moduli space of heterotic string compactifications \cite{AGS,COM,OssaSvanes,MelSha}. In this setup, a remarkable observation in \cite{AOMSS} is that the deformation theory for pairs given by a Calabi-Yau manifold and a holomorphic string algebroid relates to the \emph{heterotic superpotential} and \emph{Yukawa couplings} for the heterotic string. Similar results for the $G_2$ Hull-Strominger system are provided in \cite{Larfors}. 

The structure and main results of the paper are as follows. In Section \ref{sec:def-basic}, we introduce (holomorphic) string algebroids and their morphisms, we interpret them in terms of Dolbeault operators on orthogonal bundles, and describe the group of automorphisms.
Section \ref{sec:holomorphiccase} deals with the classification of string algebroids in terms of \v Cech cohomology. This is a technically involved problem but we give a detailed and satisfactory answer in Theorem \ref{th:classification}. 
As a consequence, in Proposition \ref{theo:deRhamC} we prove that the parameter space for $P$ fixed becomes a torsor for a quotient of  $H^1(\Omega^{2,0}_{cl})$--the space of isomorphism classes of exact holomorphic Courant algebroids--by a subgroup $I_{[P]}$ depending on the holomorphic gauge group of $P$. The quotient $H^1(\Omega^{2,0}_{cl})/I_{[P]}$ provides an analogue in the holomorphic category of the domain of $\vartheta$ in \eqref{eq:vartheta}.
In Section \ref{sec:deformation-theory}, we study the deformation theory of a string algebroid $(Q,P,\rho)$ on a fixed manifold $X$. 
In particular, in Proposition \ref{theo:MC} we construct a differential graded Lie algebra whose Maurer-Cartan equation describes the deformation theory of a string algebroid. This is used in Theorem \ref{thm:slice} to give a locally complete family of deformations, by following the method of the Kuranishi slice. 
Finally, Appendix \ref{sec:smooth} contains an analogous study of the classification of smooth string algebroids. 

\textbf{Acknowledgments:} The authors would like to thank Ra\'ul Gonz\'alez Molina, Pedram Hekmati, Johannes Huisman, Brent Pym, Ping Xu and Xiaomeng Xu for helpful conversations.
Part of this work was undertaken during visits of MGF to the IMS at ShanghaiTech University, the Fields Institute, LMBA and IMPA, of RR to ICMAT and of CT to IMPA, CIRGET, ICMAT and the Weizmann Institute. We would like to thank these very welcoming institutions for the support and for providing a nice and stimulating working environment.

\section{Definition and basic properties}
\label{sec:def-basic}

In this section we introduce our main object of study: string algebroids, and their morphisms. Roughly speaking, these are certain Courant extensions \cite{Bressler} 
of the Atiyah Lie algebroid of a holomorphic principal bundle. Their morphisms correspond to the \emph{inner morphisms} of the underlying Courant algebroid, as introduced by \v Severa \cite[Letter 4]{Severa}. 

\subsection{Definition of string algebroids}\label{subsec:string}

We start by recalling some basic properties about holomorphic Courant algebroids, following \cite{ChStXu,GrSt}. Let $X$ be a complex manifold of dimension $n$. We denote by $\cO_X$ and $\underline{\CC}$ the sheaves of holomorphic functions and $\CC$-valued constant functions on $X$, respectively.

\begin{definition}\label{def:Courant}
A holomorphic Courant algebroid $(Q,\la\cdot,\cdot\ra,[\cdot,\cdot],\pi)$ over $X$ consists of a holomorphic vector bundle $Q \to X$, with sheaf of sections also denoted by $Q$, together with a holomorphic non-degenerate symmetric bilinear form 
$$
\la\cdot,\cdot\ra \in H^0(Q^*\otimes Q^*),
$$
a holomorphic vector bundle morphism $\pi:Q\to TX$ called anchor map, and a Dorfman bracket, that is, a homomorphism of sheaves of $\underline{\CC}$-modules
$$
[ \cdot,\cdot ] \colon Q \otimes_{\underline{\CC}} Q \to Q,
$$
satisfying, for $u,v,w\in Q$ and $\phi\in \cO_X$,
  \begin{itemize}
  \item[(D1):] $[u,[v,w]] = [[u,v],w] + [v,[u,w]]$,
  \item[(D2):] $\pi([u,v])=[\pi(u),\pi(v)]$,
  \item[(D3):] $[u,\phi v] = \pi(u)(\phi) v + \phi[u,v]$,
  \item[(D4):] $\pi(u)\la v, w \ra = \la [u,v], w \ra + \la v,
    [u,w] \ra$,
  \item[(D5):] $[u,v]+[v,u]=2\pi^* d\la u,v\ra$.
  \end{itemize}
\end{definition}

We shall denote a holomorphic Courant algebroid $(Q,\la\cdot,\cdot\ra,[\cdot,\cdot],\pi)$ simply by~$Q$. Using the isomorphism
$$
\langle \cdot,\cdot \rangle \colon Q \to Q^*
$$
we obtain a holomorphic sequence
\begin{equation*}\label{eq:holCouseq}
  \xymatrix{
    T^*X \ar[r]^{\pi^*} & Q \ar[r]^{\pi} & TX.
  }
\end{equation*}
The subsheaf $(\Ker \pi)^\perp \subset Q$ generated by the local sections orthogonal to $\Ker \pi$ coincides with $\pi^*(T^*X)$, so it is coisotropic, that is,
$$
(\Ker \pi)^\perp  \subset \Ker \pi.
$$
The sheaves $(\Ker \pi)^\perp$ and $\Ker \pi$ are two-sided ideals of $Q$ with respect to the Dorfman bracket $[\cdot,\cdot]$.

A holomorphic Courant algebroid $Q$ is said to be transitive when the anchor map $\pi \colon Q \to TX$ is surjective. In this case, both $(\Ker \pi)^\perp$ and $\Ker \pi$ are locally free and the quotient $$
A_Q = Q/(\Ker \pi)^\perp
$$ 
is a vector bundle inheriting the structure of a holomorphic Lie algebroid. Furthermore, the holomorphic subbundle
\begin{equation*}\label{eq:adQ}
\ad_Q = \Ker \pi/(\Ker \pi)^\perp \subset A_Q
\end{equation*}
becomes a holomorphic bundle of quadratic Lie algebras $(\ad_Q,\la\cdot,\cdot\ra)$.

We introduce next the notion of a string algebroid. Intuitively, a string algebroid is a transitive holomorphic Courant algebroid $Q$ such that $A_Q$ and $\ad_Q$ are determined by a holomorphic principal bundle and a choice of bi-invariant pairing in the Lie algebra of the structure group.

To be more precise, let $G$ be a complex Lie group. Let $p \colon P \to X$ be a holomorphic principal $G$-bundle over $X$. The holomorphic Atiyah Lie algebroid $A_{P}$ of $P$ has underlying holomorphic bundle
$$
TP/G \to X,
$$
with local sections given by $G$-invariant holomorphic vector fields on $X$, anchor map $dp \colon TP/G \to TX$, and bracket induced by the Lie bracket on $TP$.  The holomorphic bundle of Lie algebras $\Ker dp \subset A_{P}$ fits into the short exact sequence of holomorphic Lie algebroids
$$
0 \to \Ker dp \to A_{P} \to TX \to 0
$$
and is isomorphic to the adjoint bundle,
$$
\Ker dp \cong \ad P = P \times_{G} \mathfrak{g},
$$
induced by the adjoint representation of $G$ on its Lie algebra $\mathfrak{g}$.

 From now on, we fix a non-degenerate bi-invariant symmetric bilinear form
$$
c \colon \mathfrak{g} \otimes \mathfrak{g} \to \mathbb{C}.
$$
Using the pairing $c$, the adjoint bundle inherits naturally the structure of a holomorphic bundle of quadratic Lie algebras $(\ad P,c)$.


\begin{definition}\label{def:stringholCour}
A string algebroid with structure group $G$ and underlying pairing $c$ is a tuple $(Q,P,\rho)$ consisting of:
\begin{itemize}
	\item a transitive holomorphic Courant algebroid $Q$,
	\item a holomorphic principal $G$-bundle $P$,
	\item a holomorphic bracket-preserving morphism $\rho:Q\to A_P$ fitting into a short exact sequence
	\begin{equation}\label{eq:defstring}
	\xymatrix{
		0 \ar[r] & T^*X \ar[r] & Q \ar[r]^\rho & A_P \ar[r] & 0
	}
	\end{equation}
	such that the induced map of holomorphic Lie algebroids $\rho \colon A_Q \to A_P$ is an isomorphism restricting to an isomorphism $(\ad_Q,\la\cdot,\cdot\ra ) \cong (\ad P,c)$. 
\end{itemize}
\end{definition}

Note that the property of being a bracket-preserving morphism implies that $\rho$ is compatible with the anchor maps, in the sense that $dp \circ \rho = \pi$. In other words, $\rho$ is a morphism of Leibniz algebroids. Consequently, when $G$ is trivial, a string algebroid is simply an exact holomorphic Courant algebroid as in \cite{G2}.


The following definition of morphism between string algebroids will be key to the present work.

\begin{definition}\label{def:stringholCourmor}
A morphism of string algebroids from $(Q,P,\rho)$ to $(Q',P',\rho')$ is a pair $(\varphi,g)$, where $\varphi \colon Q \to Q'$ is a morphism of holomorphic Courant algebroids and $g \colon P \to P'$ is an homomorphism of holomorphic principal bundles covering the identity on $X$, such that the following diagram is commutative
\begin{equation}\label{eq:defstringiso}
    \xymatrix{
      0 \ar[r] & T^*X \ar[r] \ar[d]^{id} & Q \ar[r]^\rho \ar[d]^{\varphi} & A_P \ar[r] \ar[d]^{g} & 0,\\
      0 \ar[r] & T^*X \ar[r] & Q' \ar[r]^{\rho'} & A_{P'} \ar[r] & 0.
    }
\end{equation}
We say that $(Q,P,\rho)$ is isomorphic to $(Q',P',\rho')$ if there exists a morphism $(\varphi,g)$ such that $\varphi$ and $g$ are isomorphisms.
\end{definition}

Given a string algebroid $(Q,P,\rho)$ we say that $P$ is the underlying principal bundle of $Q$, or that $Q$ is a Courant algebroid extension of $P$ (cf. \cite{Bressler}). We say that $(Q,P,\rho)$ and $(Q',P,\rho')$ are equivalent as Courant extensions if they are isomorphic and further $g = id$ in \eqref{eq:defstringiso}. 

\subsection{String algebroids and Dolbeault operators}\label{sec:stringDolb}

We study next string algebroids from the point of view of differential geometry, in terms of smooth vector bundles and Dolbeault operators. Our discussion follows closely the classification of regular Courant algebroids in \cite{ChStXu} (see also \cite[Section 3]{GF}). 
We denote by $\Omega^{p,q}$ the space of differential forms on $X$ of type $(p,q)$. Given a smooth complex vector bundle $E$, we denote by $\Omega^{p,q}(X,E)$ (or $\Om^{p,q}(E)$) the space of smooth $E$-valued differential forms on $X$ of type $(p,q)$. 

Let $(Q,P,\rho)$ be a string algebroid on $X$. 
We denote by $\underline{Q}$ the smooth complex vector bundle underlying $Q$ and by $\underline{P}$ the smooth principal $G$-bundle underlying $P$. 
We choose a smooth isotropic splitting 
$$
\lambda \colon T^{1,0}X \to \underline{Q}.
$$
Let $\theta$ be the unique connection on $\underline{P}$ such that 
\begin{equation}\label{eq:thetalambda}
\theta^\perp = \rho \circ \lambda \qquad \textrm{and} \qquad P_\theta = (\underline{P},\theta^{0,1}) = P,
\end{equation}
where $\theta^\perp \colon T^{1,0}X \to A_P$ is the lifting determined by $\theta$. Note that the curvature $F_\theta$ of $\theta$ satisfies 
$$
F_\theta^{0,2} = 0.
$$
This data determines a vector bundle isomorphism
\begin{equation}\label{eq:Qsmooth}
\varphi_\lambda \colon \underline Q_0 = T^{1,0}X \oplus \ad \underline{P} \oplus (T^{1,0}X)^* \to \underline{Q} 
\end{equation}
defined by 
$$
\varphi_\lambda(V + r + \xi) = \rho_{|\lambda^\perp}^{-1}(\theta^\perp V + r) + \frac{1}{2}\pi^*\xi,
$$
where $\lambda^\perp \subset \underline Q$ is the orthogonal complement to the image of $\lambda$, and $\pi^* \colon (T^{1,0}X)^* \to \underline Q$ denotes the composition of the dual of the anchor map $\pi$ with the isomorphism $\underline Q^* \to \underline Q$ given by $\la\cdot,\cdot\ra$. Via $\varphi_\lambda$, the induced pairing is
\begin{equation}\label{eq:holpairing}
\langle V + r + \xi , V + r + \xi \rangle_0 = \xi(V) + c(r,r),
\end{equation}
the induced anchor map is
\begin{equation}
\label{eq:anchor}
 \pi_0(V+ r + \xi) = V,
\end{equation}
and we define
\begin{equation}\label{eq:rho0}
\rho_0 := \rho \circ \varphi_\lambda \colon \underline Q_0 \to A_P \colon V + r + \xi \mapsto \theta^\perp V + r.
\end{equation}
It is obvious that $(Q_0,P,\rho_0)$, where $Q_0$ is $\underline{Q}_0$ endowed with the Dolbeault operator $\dbar_0$ induced by $\varphi_\lambda$, the pairing $\la\cdot,\cdot\ra_0$, and the induced Dorfman bracket $[\cdot,\cdot]_0$, is a holomorphic string algebroid, which is isomorphic to $(Q,P,\rho)$ via the following diagram 
\begin{equation}\label{eq:Q0}
    \xymatrix{
      0 \ar[r] & T^*X \ar[r] \ar[d]^{id} & Q_0 \ar[r]^{\rho_0} \ar[d]^{\varphi_\lambda} & A_P \ar[r] \ar[d]^{\Id} & 0,\\
      0 \ar[r] & T^*X \ar[r] & Q \ar[r]^{\rho} & A_{P} \ar[r] & 0.
    }
\end{equation}
Our next goal is to express $\dbar_0$ and $[\cdot,\cdot]_0$ in terms of explicit data in the complex manifold $X$. The holomorphicity of the pairing $\langle \cdot,\cdot \rangle_0$ implies that
$$
\dbar \langle q_1 , q_2 \rangle_0 = \langle \dbar_0 q_1 , q_2 \rangle_0 + \langle q_1 , \dbar_0 q_2 \rangle_0
$$
for any pair of smooth sections $q_1,q_2 \in \Omega^0(\underline{Q}_0)$.  Combined with the holomorphicity of the top exact sequence in \eqref{eq:Q0}, it follows that
\begin{equation}\label{eq:DolbQbis}
\dbar_0 = \left(
\begin{array}{ccc}
\dbar & 0 & 0 \\
F_\theta^{1,1} & \dbar^\theta & 0 \\
H^{2,1} & 2c(F_\theta^{1,1},\cdot) & \dbar
\end{array}\right),
\end{equation}
for a suitable choice of $(2,1)$-form $H^{2,1} \in \Omega^{2,1}$, where $\dbar^\theta$ is the Dolbeault operator on $\ad \underline{P}$ induced by $\theta^{0,1}$. More explicitly
$$
\dbar_0 (V + r + \xi) = \dbar V + i_V F_\theta^{1,1} + \dbar^\theta r + \dbar \xi + i_{V}H^{2,1} + 2c(F_\theta^{1,1},r).
$$
Finally, 
arguing as in \cite[Thm. 2.3]{ChStXu}, we deduce the following explicit form for the bracket
\begin{equation}\label{eq:bracket}
\begin{split}
[V+ r + \xi,W + t + \eta]_0  = {}& [V,W] - F^{2,0}_\theta(V,W) + \partial^\theta_V t - \partial^\theta_W r - [r,t]\\
& + i_V \partial \eta + \partial (\eta(V)) - i_W\partial \xi + i_Vi_W H^{3,0}\\
& + 2c(\partial^\theta r,t) + 2c(i_V F_\theta^{2,0},t) - 2c(i_W F_\theta^{2,0},r)
\end{split}
\end{equation}
for $H^{3,0} \in \Omega^{3,0}$ defined by
\begin{equation}\label{eq:H30}
H^{3,0}(V,W,Z) = - 2 \la [\lambda(V),\lambda(W)],\lambda(Z)\ra.
\end{equation}

\begin{proposition}\label{prop:strholCour}
Let $(Q,P,\rho)$ be a string algebroid on $X$. Then, a choice of isotropic splitting $
\lambda \colon T^{1,0}X \to \underline{Q}$ determines a connection $\theta$ on $\underline{P}$ and $H = H^{3,0} + H^{2,1} \in \Omega^{3,0} \oplus \Omega^{2,1}$, via \eqref{eq:thetalambda}, \eqref{eq:DolbQbis} and \eqref{eq:H30}, such that
\begin{equation}\label{eq:int2}
dH + c(F_\theta \wedge F_\theta) = 0.
\end{equation}
Furthermore, the string algebroids $(Q,P,\rho)$ and $(Q_0,P,\rho_0)$ are isomorphic via \eqref{eq:Q0}. Conversely, given a triple $(P,H,\theta)$, where $P$ is a holomorphic $G$-bundle over $X$, $\theta$ is a connection on $\underline{P}$ such that $P_\theta = P$, and $H \in \Omega^{3,0} \oplus \Omega^{2,1}$ satisfying \eqref{eq:int2}, there is a string algebroid $(Q_0,P,\rho_0)$, with underlying smooth bundle
$$
\underline Q_0 = T^{1,0}X \oplus \ad \underline{P} \oplus (T^{1,0}X)^*,
$$
pairing \eqref{eq:holpairing}, anchor \eqref{eq:anchor}, Dolbeault operator \eqref{eq:DolbQbis}, bracket \eqref{eq:bracket}, and $\rho_0$ given by \eqref{eq:rho0}.
\end{proposition}

\begin{proof}
Given $(Q,P,\rho)$, $\lambda$ determines $(H,\theta)$ and the isomorphic string algebroid $(Q_0,P,\rho_0)$ as above. By \eqref{eq:thetalambda}, $F_\theta^{0,2} = 0$, and therefore \eqref{eq:int2} is equivalent to
\begin{equation}\label{eq:integrab}
\begin{split}
\partial H^{3,0} + c(F_\theta^{2,0} \wedge F_\theta^{2,0}) & = 0,\\
\dbar H^{3,0} + \partial H^{2,1} + 2 c(F_\theta^{2,0} \wedge F_\theta^{1,1}) & = 0,\\
\dbar H^{2,1} + c(F_\theta^{1,1} \wedge F_\theta^{1,1}) & = 0.\\
\end{split}
\end{equation}
The last condition is implied by $\dbar_0 \circ \dbar_0 = 0$. The second condition follows from the fact that $[\cdot,\cdot]_0$ defines a homomorphism of sheaves of $\underline{\CC}$-modules
$$
[ \cdot,\cdot ]_0 \colon \cO_{Q_0} \otimes_{\underline{\CC}} \cO_{Q_0} \to \cO_{Q_0},
$$
while the first condition is implied by (D1) in Definition \ref{def:Courant}.

For the converse, let $P$, $\theta$ and $H$ as in the statement, satisfying \eqref{eq:int2}. Then, the last condition in \eqref{eq:integrab} is equivalent to $\dbar_0$ in \eqref{eq:DolbQbis} defining a holomorphic extension of the form \eqref{eq:defstring} with a homomorphism of sheaves of $\cO_X$-modules
$$
\langle \cdot,\cdot \rangle_0 \colon \cO_{Q_0} \otimes_{\cO_X} \cO_{Q_0} \to \cO_X
$$
induced by \eqref{eq:holpairing}. The second condition ensures that $[\cdot,\cdot]$ defines a homomorphism of sheaves of $\underline{\CC}$-modules $
[ \cdot,\cdot ]_0 \colon \cO_{Q_0} \otimes_{\underline{\CC}} \cO_{Q_0} \to \cO_X$, while the first condition implies that (D1) in Definition \ref{def:Courant} is satisfied. The rest of axioms in Definition \ref{def:stringholCour} are satisfied by construction.
\end{proof}

As a consequence of Proposition \ref{prop:strholCour}, we obtain a necessary condition for a holomorphic principal bundle $P$ to admit a Courant extension of the form \eqref{eq:defstring}, namely,
\begin{equation}\label{eq:p1dR}
p_1^c(P) = 0 \in H^4(X,\CC).
\end{equation}
Here $p_1^c(P)$ denotes the first Pontryagin class of $P$ with respect to the bi-invariant symmetric pairing $c$ on $\mathfrak{g}$. Condition \eqref{eq:p1dR} boils down to the fact that a string algebroid $Q$ has an associated smooth complex Courant algebroid $\underline{Q} \oplus T^{0,1}X \oplus  (T^{0,1}X)^*$ \cite{GrSt}, combined with the classification results for transitive Courant algebroids in \cite{Bressler,ChStXu,Severa}. A refined version of this obstruction will be considered in Section \ref{subsec:1stcoho}.

To finish this section, we use Proposition \ref{prop:strholCour} to show some interesting examples of string algebroids on Calabi-Yau threefolds (not necessarily algebraic). Most of these examples have been intensively studied in the algebro-geometric and physics literature, motivated by the search of stable bundles with prescribed Chern classes in the study of heterotic string compactifications.

\begin{example}\label{example}
Let $\{V_j\}_{j=0}^k$ be a collection of holomorphic vector bundles over a complex manifold $X$ such that there exists complex numbers $(\mu_0,\ldots,\mu_k)$, not all zero, such that 
\begin{equation}\label{eq:Ch2}
\sum_{j=0}^k \mu_j \textrm{ch}_2(V_j) = 0 \in H^{2,2}_{BC}(X),
\end{equation}
where $H^{p,q}_{BC}(X)$ denote the Bott-Chern cohomology groups of $X$ (cf. \cite[Cor. 4.2]{Bressler}), and $\textrm{ch}_2(V_j)$ denotes the second Chern character of $V_j$ (see e.g. \cite{Kob}). Then, for any choice of hermitian metrics $h_j$ on $V_j$ there exists $\tau \in \Omega^{1,1}$ such that
$$
2i\partial \dbar \tau = \sum_{j=0}^k \mu_j\tr F_{h_j} \wedge F_{h_j},
$$
where $F_{h_j}$ denotes the curvature of the Chern connection of $h_j$. Let $P$ denote the holomorphic bundle of split frames of $\oplus_{j=0}^k V_j$, and consider the symmetric bilinear form
$$
c = \sum_{j=0}^k \mu_j\tr_{\mathfrak{gl}(r_j,\CC)}
$$
on the Lie algebra of the structure group, where $r_j$ denotes the rank of $V_j$. Then, by choosing $H = 2i \partial \tau \in \Omega^{2,1}$ and $\theta$ the Chern connection of $h_0 \times \ldots\times h_k$, we obtain a Courant extension of $P$, and hence a string algebroid, by applying Proposition \ref{prop:strholCour}. Explicit constructions of such bundles for $X$ a compact Calabi-Yau threefold with $k = 1$, $V_0 = TX$ and $V_1$ not isomorphic to $V_0$ can be found in \cite{AndHo2,DouglasZhou,Huyb,Jardim,LWY} and references therein.
\end{example}

\begin{remark}
In the physics literature one often finds \eqref{eq:Ch2} replaced by the more general condition
\begin{equation*}\label{eq:Ch2bis}
\sum_{j=0}^k \mu_j \textrm{ch}_2(V_j) = [W],
\end{equation*}
where $[W]$ is the dual of the homology class of a holomorphic curve on the threefold $X$. Provided that $X$ is compact and algebraic, the class $[W]$ can be realized as the second Chern character of a holomorphic vector bundle over $X$ (see \cite[Thm. 11.32]{Voisin}), and thus it falls within the situation considered in Example \ref{example}. Nonetheless, we expect that the right physical treatment of this class of examples uses the twisted Courant algebroids in \cite{Pym}. 
\end{remark}

\subsection{Isomorphism classes}\label{subsec:class}

We provide next a more amenable description of the set of isomorphism classes of string algebroids, following Definition \ref{def:stringholCourmor}. Since any string algebroid is isomorphic to one of the form $(Q_0,P,\rho_0)$, as in Proposition \ref{prop:strholCour}, we can study the equivalence relation in terms of the triples $(P,H,\theta)$.

\begin{lemma}\label{lemma:strisomorphism}
Let $(P,H,\theta)$ and $(P',H',\theta')$ be as in Proposition \ref{prop:strholCour}. Then, the associated string algebroids $(Q_0,P,\rho_0)$ and $(Q_0',P',\rho_0')$ are isomorphic if and only if there exists an isomorphism $g \colon P \to P'$ of holomorphic principal $G$-bundles such that 
\begin{equation}\label{eq:anomalyiso}
H'  = H - 2c(a,F_\theta) - c(a,d^\theta a) - \frac{1}{3} c(a,[a,a]) - dB
\end{equation}
for some $B\in \Omega^{2,0}$, where $a = g^{-1}\theta' - \theta \in \Omega^{1,0}(\ad \underline P)$.
\end{lemma}

\begin{proof}
Assume that $(\varphi,g)$ is an isomorphism between $(Q_0,P,\rho_0)$ and $(Q_0',P',\rho_0')$. In particular, $g \colon P \to P'$ is an isomorphism of holomorphic principal bundles, and therefore $g \theta^{0,1} = \theta'^{0,1}$, and
$$
\varphi \colon \underline Q_0 \to \underline Q_0
$$
is an orthogonal bundle isomorphism. Since $\varphi$ is the identity on $(T^{1,0}X)^*$ it follows that (see \cite[Sec. 4.5]{grt})
$$
\varphi = f_\nu(B,a)
$$
where
$$
f_\nu:=\left(
\begin{array}{ccc}
  \Id & 0 & 0 \\
  0 & \nu & 0  \\
  0 & 0 & \Id 
\end{array}
\right).
$$ 
for $\nu \colon \ad \underline{P} \to \ad \underline{P}'$ an orthogonal bundle isomorphism covering the identity on $X$ and preserving the Lie bracket, and
$$
(B,a):=\left(
  \begin{array}{ccc}
    \Id & 0 & 0 \\
    a & \Id & 0 \\
    B-c(a,a)  & -2c(a,\cdot) & \Id 
  \end{array}
\right),
$$ 
for a suitable $B\in \Omega^{2,0}$. More explicitly, $\varphi$ acts on sections of $\underline Q_0$ by
\begin{equation}\label{eq:action-of-the-automorphism}
\begin{split}
\varphi(V+r+\xi) = {} & V + \nu(r+i_V a) \\ & + \xi + i_V (B-c(a,a))-2c(a,r). \end{split}
\end{equation}
The condition $\rho_0^{\theta'} \circ \varphi  = g \circ \rho_0^\theta$ implies, in particular, (see \eqref{eq:rho0})
$$
\rho_0^{\theta'} \circ \varphi(r) = \nu r = gr = g \circ \rho_0^\theta(r)
$$
for all $r \in \ad \underline{P}$, and therefore $\nu = g$. Here, we abuse of the notation and denote by $g$ the isomorphism $\ad \underline{P} \to \ad \underline{P}'$ induced by $dg \colon TP \to TP'$. Furthermore, we have
\begin{align*}
\rho_0^{\theta'} \circ \varphi(V) & = \theta'^\perp(V) + g(i_V a)\\
& = g\theta^\perp(V) + g(i_Va ) + i_V(g\theta - \theta')\\
& = g \circ \rho_0^\theta(V) + i_Vg(a + \theta - g^{-1}\theta')
\end{align*}
for any $V \in T^{1,0}X$, and therefore $\rho_0^{\theta'} \circ \varphi  = g \circ \rho_0^\theta$ implies $a = g^{-1}\theta' - \theta$. Notice that $g \theta^{0,1} = \theta'^{0,1}$ implies $a \in \Omega^{1,0}(\ad \underline P)$, as claimed. Finally, by the proof of \cite[Prop. 4.3]{grt}, the fact that $\varphi$ is holomorphic and bracket preserving implies that \eqref{eq:anomalyiso} is satisfied.
\end{proof}

\begin{remark}\label{rem:notationwedge}
In formula \eqref{eq:anomalyiso} we abuse of the notation and omit the wedge product of differential forms. This simplified notation will be used in the rest of the paper, when there is no possibility of confusion.
\end{remark}

Given an integer $k \geqslant 0$, consider the following subspace of the complex $(k+2)$-forms on $X$
\begin{equation}\label{eq:Omegaleq}
\Omega^{\leqslant k} = \oplus_{j\leqslant k} \Omega^{j+ 2,k-j},
\end{equation}
with the convention that $\Omega^{p,q} = 0$ if $p < 0$ or $q <0$. Explicitly, for $0 \leqslant k \leqslant 2$,
\begin{align*}
\Omega^{\leqslant 0} & = \Omega^{2,0},\\
\Omega^{\leqslant 1} & = \Omega^{3,0} \oplus \Omega^{2,1},\\
\Omega^{\leqslant 2} & = \Omega^{4,0} \oplus \Omega^{3,1} \oplus \Omega^{2,2}.
\end{align*}
Given a holomorphic principal $G$-bundle $P$ over $X$ we denote by $\mathcal{A}_{P}$ the space of connections on $\underline{P}$ such that $F_\theta^{0,2} = 0$ and $P_\theta = P$. As a direct consequence of Proposition \ref{prop:strholCour} and Lemma \ref{lemma:strisomorphism} we obtain the following result.

\begin{proposition}\label{prop:classification}
There is a one to one correspondence between isomorphism classes of string algebroids and elements in
\begin{equation*}\label{eq:lescEind3}
\{(P,H,\theta): (H,\theta) \in \Omega^{\leqslant 1} \times \cA_{P} \; | \; dH + \la F_\theta \wedge F_\theta \ra = 0\}/ \sim,
\end{equation*}
where $(P,H,\theta)\sim (P',H',\theta')$ if there exists an isomorphism $g \colon P \to P'$ of holomorphic principal $G$-bundles such that \eqref{eq:anomalyiso} holds for some $B\in \Omega^{\leqslant 0}$.
\end{proposition}

\subsection{Automorphisms and the Chern-Simons three-form}\label{subsec:Aut}

Let $(Q,P,\rho)$ be a string algebroid. Let $\Aut Q$ denote the group of automorphisms of the holomorphic Courant algebroid $Q$, that is, automorphisms of the holomorphic vector bundle $Q$ that preserve the bracket and the pairing. Let $\cG_P$ denote the gauge group of $P$, given by holomorphic $G$-bundle automorphisms $g \colon P \to P$ covering the identity on $X$. By Definition \ref{def:stringholCourmor}, the group of automorphisms $\Aut (Q,P,\rho)$ is given by pairs $(\varphi,g) \in \Aut Q \times \cG_P$ such that the following diagram is commutative
\begin{equation}\label{eq:defstrinaut}
    \xymatrix{
      0 \ar[r] & T^*X \ar[r] \ar[d]^{id} & Q \ar[r]^\rho \ar[d]^{\varphi} & A_P \ar[r] \ar[d]^{g} & 0,\\
      0 \ar[r] & T^*X \ar[r] & Q \ar[r]^{\rho} & A_{P} \ar[r] & 0.
    }
\end{equation}
Consequently, there is an homomorphism of groups
\begin{equation}\label{eq:AutQGP}
\Aut (Q,P,\rho) \to \cG_P \colon (\varphi,g) \mapsto g.
\end{equation} 
The rest of this section is devoted to describe the Kernel and the image of \eqref{eq:AutQGP}. For this, we first obtain an explicit characterization of the automorphism group of the string algebroid $(Q_0,P,\rho_0)$ associated to a triple $(P,H,\theta)$, as in Proposition \ref{prop:classification}. Following the proof of Lemma \ref{lemma:strisomorphism}, an automorphism $(\varphi,g)$ of $(Q_0,P,\rho_0)$ is of the form
$$
\varphi = f_g(B,a^g),
$$
for $a^g:= g^{-1}\theta - \theta$, acting on sections of $\underline Q_0$ by
\begin{equation*}
\begin{split}
\varphi(V+r+\xi) = {} & V + g(r+i_V a^g) \\ & + \xi + i_V (B-c(a^g,a^g))-2c(a^g,r). \end{split}
\end{equation*}

\begin{lemma}\label{lemma:Autstring}
Let $(P,H,\theta)$ be as in Proposition \ref{prop:classification}, and $(Q_0,P,\rho_0)$ its associated string algebroid. Then, elements in $\Aut(Q_0,P,\rho_0)$ correspond to transformations $f_g(B,a^g)$, with $g \in \cG_P$, $B \in \Omega^{2,0}$ and $a^g:= g^{-1}\theta - \theta$, which satisfy
\begin{equation*}\label{eq:AutEmathring2}
dB = 2c(a^g,F) + c(a^g,d^\theta a^g) + \frac{1}{3} c(a^g,[a^g,a^g]).
\end{equation*}
The group structure is given by
$$ 
(g,B)(g',B') = (gg',B+B'+c((g'^{-1} \cdot a^g)\wedge a^{g'})),
$$
with Lie algebra $\Lie \; \Aut(Q,P,\rho) \subset \Om^{0}(\ad \underline P) \oplus
\Om^{2,0}$ given by
$$
\Lie \; \Aut(Q,P,\rho) = \{\alpha + b : \; \dbar^\theta \alpha = 0, \; d(b - 2c (\alpha F_\theta)) = 0\}.
$$
\end{lemma}

As an immediate consequence of Lemma \ref{lemma:Autstring} we obtain that $\Aut (Q,P,\rho)$ fits into an exact sequence of groups
\begin{equation*}
    \label{eq:AutW}
    \xymatrix{
      0 \ar[r] & \Omega^{2,0}_{cl} \ar[r] & \Aut (Q,P,\rho) \ar[r] & \cG_P\\
    }
\end{equation*}
where $\Omega^{2,0}_{cl}$ denotes the space of closed $(2,0)$-forms on $X$. Our next goal is to characterize the image of \eqref{eq:AutQGP}. We start by recalling some generalities about the Chern-Simons three-form. Let $\underline{P}$ be the smooth principal $G$-bundle over $X$ underlying $P$. Given a connection $\theta$ on $\underline{P}$, the Chern-Simons three-form of $\theta$ is a $G$-invariant differential form of degree three on the total space of $\underline{P}$ defined by (see Remark \ref{rem:notationwedge})
\begin{equation*}\label{eq:CS}
CS(\theta) = -  \frac{1}{6}c(\theta,[\theta,\theta]) + c(F_\theta , \theta) \in \Omega^3(\underline{P}),
\end{equation*}
where $F_\theta \in \Omega^2(\ad \underline{P})$ is the curvature of $\theta$. Note that
\begin{equation}\label{eq:diff-of-CS}
dCS(\theta)=c(F_\theta\wedge F_\theta),
\end{equation}
and that $CS(\theta)$ is functorial, in the sense that
\begin{equation}\label{eq:CSfunctorial}
g^*CS(\theta) = CS(g^* \theta)=CS(g^{-1}\theta)
\end{equation}
for any automorphism $g$ of $\underline P$. Recall that, given another connection $\theta' = \theta + a$ on $\underline{P}$, the following defines a basic three-form on the principal bundle $\underline{P}$:
\begin{equation}\label{eq:CSdiff}
CS(\theta') - CS(\theta) - dc(\theta' \wedge \theta). 
\end{equation}
By setting $a = \theta' - \theta \in \Omega^1(\ad \underline{P})$, this form is the pullback to $\underline{P}$ of
\begin{equation}\label{eq:CSdiffexp}
2c(a \wedge F_\theta) + c(a \wedge d^\theta a) + \frac{1}{3}c(a,[a,a]) \in \Omega^3
\end{equation}
via the natural projection $\underline{P} \to X$. In particular, following the notation in Proposition \ref{prop:classification}, if $\theta, \theta' \in \cA_P$ we have that \eqref{eq:CSdiffexp} lies in $\Omega^{\leqslant 1}$.

Consider now the complex defined by \eqref{eq:Omegaleq} and the usual exterior de Rham differential
\begin{equation}\label{eq:cxleqk}
  \xymatrix@R-2pc{
\ldots \ar[r]^d & \Omega^{\leqslant k} \ar[r]^d & \Omega^{\leqslant k+1} \ar[r]^d & \ldots
}
\end{equation}
It is easy to verify that $(\Omega^{\leqslant\bullet},d)$ is elliptic. We denote by $H^k(\Omega^{\leqslant\bullet})$ its $k$-th cohomology group.

\begin{lemma}\label{lemma:sigmaP}
Let $P$ be a holomorphic principal $G$-bundle over $X$. Then, there is an homomorphism of groups
\begin{equation}\label{eq:sigmaP}
\sigma_P \colon \cG_P \to H^1(\Omega^{\leqslant\bullet})
\end{equation}
defined by 
$$
\sigma_P(g) = [CS(g \theta) - CS(\theta) - d c (g \theta \wedge \theta)] \in H^1(\Omega^{\leqslant\bullet}),
$$
for any choice of connection $\theta \in \cA_P$.
\end{lemma}
\begin{proof}
For a choice of connection $\theta \in \mathcal{A}_{P}$, given $g \in \cG_P$ we can associate a closed three-form on $X$ (see \eqref{eq:diff-of-CS}):
$$
\sigma^\theta(g) = CS(g \theta) - CS(\theta) - d c (g \theta \wedge \theta) \in \Omega^{\leqslant 1}.
$$
Given $\theta' \in \mathcal{A}_{P}$, functoriality of the Chern-Simons three-form \eqref{eq:CSfunctorial} gives
$$
\sigma^{\theta'}(g) - \sigma^\theta(g) = d c (g \theta^{1,0} \wedge \theta^{1,0}) - d c (g \theta'^{1,0} \wedge \theta'^{1,0}) \in d \Omega^{\leqslant 0},
$$
and therefore the class of $\sigma^\theta(g)$ in $H^1(\Omega^{\leqslant \bullet})$ is independent of the choice of connection $\theta$. Finally,
\begin{align*}
\sigma^\theta(gg') 
& = g_*(CS(g' \theta) - CS(\theta) - d c (g' \theta \wedge \theta))\\
& \phantom{ {} = } + (CS(g \theta) - CS(\theta) - d c (g\theta \wedge \theta))\\
& \phantom{ {} = } - d c (gg' \theta \wedge \theta) + d c (g' \theta \wedge \theta) + d c (g \theta \wedge \theta)\\
& = \sigma^\theta(g) + \sigma^\theta(g) + dB
\end{align*}
for $B \in \Omega^{\leqslant 0}$. For the first equality we have used that \eqref{eq:CSdiff} is basic, and therefore invariant by pullback by elements $g \in \cG_P$. The second equality follows by \eqref{eq:afg-af-ag} and the functoriality of the Chern-Simons three-form \eqref{eq:CSfunctorial}.
\end{proof}

We are now ready to state the main result of this section. The proof is a straightforward consequence of Lemma \ref{lemma:Autstring} and Lemma \ref{lemma:sigmaP}.

\begin{proposition}\label{prop:AutQ}
There is an exact sequence of groups
\begin{equation*}
    \label{eq:AutW}
    \xymatrix{
      0 \ar[r] & \Omega^{2,0}_{cl} \ar[r] & \Aut (Q,P,\rho) \ar[r] & \Ker \sigma_P \ar[r] & 1
    }
\end{equation*}
where $\Omega^{2,0}_{cl}$ denotes the space of closed $(2,0)$-forms on $X$.
\end{proposition}

\begin{remark}
The group of automorphisms of a string algebroid is the analogue in the holomorphic category of the group $\mathring{\Aut} \; E$ studied in \cite[Cor. 4.2]{grt} (for automorphisms covering the identity on the base).
\end{remark}

\section{Classification}\label{sec:holomorphiccase}

Let $G$ be a complex Lie group. We assume that the Lie algebra $\mathfrak{g}$ of $G$ admits a bi-invariant symmetric bilinear form $c \colon \mathfrak{g} \otimes \mathfrak{g} \to \CC$. Let $X$ be a complex manifold. In this section we obtain our classification of string algebroids in terms of the first \v Cech cohomology of a natural holomorphic sheaf of non-abelian groups.

\subsection{The sheaf $\cS$}
\label{sec:Chern-Simons-sheaf}


Let $U \subset X$ be an open subset of $X$, regarded as a complex manifold. Let $P_0 = U \times G$ be the holomorphically trivial principal $G$-bundle on $U$. We denote by $\theta_0$ the trivial connection in the underlying smooth principal $G$-bundle $\underline{P}_0$. Consider the string algebroid $(Q_0,P_0,\rho_0)$ (see Definition \ref{def:stringholCour} and Proposition \ref{prop:strholCour}) given by the holomorphic Courant algebroid
$$
Q_0 = TU \oplus \mathfrak{g} \oplus T^*U
$$
with pairing
\begin{equation*}\label{eq:smpairing}
\langle V + r + \xi , V + r + \xi \rangle = \xi(V) + c(r,r),
\end{equation*}
anchor map
$$
\pi_{Q_0} (V + r + \xi) = V,
$$
and Dorfman bracket
\begin{equation*}\label{eq:brackethol}
\begin{split}
[V+ r + \xi,W + t + \eta] & {} = [V,W] + i_V \partial^{\theta_0} t - i_W \partial^{\theta_0} r - [r,t]\\
& {} \phantom{=}  {} + L_V \eta  - i_W d \xi + 2c(d^{\theta_0} r,t),
\end{split}
\end{equation*}
where $V,W$ are holomorphic sections of $TU$, $r,t$ are $\mathfrak{g}$-valued holomorphic maps, and $\xi,\eta$ are holomorphic $(1,0)$-forms on $U$, and map $\rho_0$ given by
$$
\rho_0(V + r +\xi) = V + r.
$$

Let $\cO_{G}(U)$ denote the space of $G$-valued holomorphic maps on $U$. Let $\Omega^{2,0}(U)$ denote the space of $(2,0)$-forms on $U$. Given $g \in \cO_{G}(U)$, we denote
$$
a^g = g^{-1} \theta_0 - \theta_0 = g^{-1} dg = g^{-1}\partial g \in \Omega^{1,0}(U,\mathfrak{g}),
$$
where $g^{-1} \theta_0$ is the connection given by the action of $g^{-1}$ on $\theta_0$. Given another $g' \in \mathcal{O}_G(U)$, denote by $g^{-1}a^{g'}$ the action of $g^{-1}$ on $a^{g'}$ induced by the adjoint representation.  Note that
\begin{align}\label{eq:afg-af-ag}
fa^f&=-a^{f^{-1}}, &  a^{fg} &= g^{-1}a^f + a^g, & a^{fgh}&=h^{-1}g^{-1}a^f + h^{-1}a^g + a^h.
\end{align}

Our next result, which is a straightforward consequence of Lemma \ref{lemma:Autstring}, describes the group of automorphisms $\Aut (Q_0,P_0,\rho_0)$. Here we give a different proof which does not rely on the results in \cite{grt}.

\begin{lemma}\label{lem:AutQ0hol}
The set of pairs $(g,B) \in \cO_{G}(U) \times \Omega^{2,0}(U)$ satisfying
\begin{equation}\label{eq:dBCShol}
dB = CS(g^{-1}\theta_0) - CS(\theta_0) - dc(g^{-1} \theta_0 \wedge \theta_0)
\end{equation}
with the product given by
\begin{equation}\label{eq:groupproduct}
(g_1,B_1)(g_2,B_2) = (g_1g_2, B_1 + B_2 + c(g_2^{-1}a^{g_1} \wedge a^{g_2})).
\end{equation} 
form a group, which we denote by $\mathcal{S}(U)$. Furthermore, there is a canonical identification $\cS(U) = \Aut (Q_0,P_0,\rho_0)$.
\end{lemma}

\begin{proof}
The set $(g,B) \in \cO_{G}(U) \times \Omega^{2,0}(U)$ is clearly closed under the product \eqref{eq:groupproduct}. We only need to check that the relation \eqref{eq:dBCShol} is compatible with \eqref{eq:groupproduct}. For this, we note that
\begin{align*}
d(B_1 + B_2 + c(g_2^{-1}a^{g_1} \wedge a^{g_2}))  &  {} = g_2^*(CS(g_1^{-1}\theta_0) - CS(\theta_0) - dc(a^{g_1} \wedge \theta_0))\\
& \phantom{{} =} + CS(g_2^{-1}\theta_0) - CS(\theta_0) - dc(a^{g_2} \wedge \theta_0)\\
& \phantom{{} =} + d c(a^{g_1g_2} \wedge a^{g_2}))\\
& {} = CS((g_1g_2)^{-1}\theta_0) - CS(\theta_0) - d c(a^{g_1g_2} \wedge \theta_0).
\end{align*}
For the first equality we have used that \eqref{eq:CSdiff} is basic, and therefore invariant by pullback by elements $g \in \mathcal{O}_G(U)$. The second equality follows by \eqref{eq:afg-af-ag} and the functoriality of the Chern-Simons three-form \eqref{eq:CSfunctorial}.
The last part of the statement is straightforward from Lemma \ref{lemma:Autstring}. 
\end{proof}

Consider $(g,B) \in \mathcal{S}(U)$. By using \eqref{eq:CSdiffexp}, the Maurer-Cartan equation
$$
d^{\theta_0}a^g = - \frac{1}{2}[a^g,a^g],
$$ and the fact that $g$ is holomorphic, condition \eqref{eq:dBCShol} can be written as
\begin{align}\label{eq:dBCSexphol}
\dbar B & {} = 0,\nonumber \\
\partial B & {} = - \frac{1}{6}c(a^g,[a^g,a^g]),
\end{align}
and therefore, in particular, $B$ is a holomorphic $(2,0)$-form. The second equation in \eqref{eq:dBCShol} can be alternatively written as
\begin{equation*}\label{eq:Severa}
\partial B = g^*\sigma_3,
\end{equation*}
where $\sigma_3 := - \frac{1}{6}c(\omega,[\omega,\omega])$ is the holomorphic Cartan $(3,0)$-form on $G$, defined in terms of the $\mathfrak{g}$-valued Maurer-Cartan holomorphic $(1,0)$-form $\omega = g^{-1}\del g$. Thus, elements in $\cS(U)$ correspond to (holomorphic) \emph{inner automorphisms} of $Q_0$ in the sense of \v Severa \cite{Severa}.
	



The family of complex groups $\cS(U)$, as $U$ varies in open subsets of the complex manifold $X$, defines a sheaf of (possibly non-abelian) complex groups $\cS$, canonically attached to the Lie group $G$ and the bilinear form $c$. It follows from Lemma \ref{lem:AutQ0hol} that one-cocycles in the \v Cech cohomology of $\cS$ can be regarded as gluing data for the construction of string algebroids over $X$ with structure group $G$. As we will see in Section \ref{subsec:holstringclass}, this cohomology classifies string algebroids up to isomorphism.

\subsection{ First cohomology of $\cS$}\label{subsec:1stcoho}

The goal of this section is to understand the \v Cech cohomology of the sheaf $\cS$, as defined in the previous section. Recall that, being a sheaf of non-abelian groups, in general its cohomology is only defined in degree $1$ and it has the structure of a pointed set. Thus, our next goal is to compute the first \v Cech cohomology group for the sheaf $\cS$. The first result shows that $\cS$ is locally modelled on an abelian extension of the group of $G$-valued holomorphic maps by the additive group of closed $(2,0)$-forms on $X$.

\begin{lemma}\label{lemma:exact-sequence-holomorphic}
Let $\Omega^{2,0}_{cl}$ be the sheaf of closed $(2,0)$-forms on $X$.  There is a short exact sequence of sheaves of groups
\begin{equation}\label{eq:sescE}
0 \to \Omega^{2,0}_{cl} \to \cS \to \cO_{G} \to 1.
\end{equation}
\end{lemma}
\begin{proof}
We work on a sufficiently small open set $U\subseteq X$. The exactness at $\Omega^{2,0}_{cl}$ follows from the injectivity of the map $B\mapsto (1,B)$. At $\cS$,  a section $(g,B) \in \cS$ is in the kernel of the projection to $\cO_G$ if and only if $g=1$, which, by \eqref{eq:dBCShol}, implies $dB=0$, i.e., $(1,B)$ lies in the image of $\Omega^{2,0}_{cl}$. Finally, at $\cO_{G}$, given $g\in \cO_{G}$, we have $dCS(g^{-1}\theta_0)=dCS(\theta_0)=0$, so the $(3,0)$+$(2,1)$-form
$$ CS(g^{-1}\theta_0)- CS(\theta_0) - dc(g^{-1}\theta_0\wedge \theta_0)$$
is closed. On an open subset it is exact, i.e., equals $dB$ for a $2$-form $B$. Consider the decomposition $B=B^{2,0}+B^{1,1}+B^{0,2}$ into $(p,q)$-forms. By the type of $dB$, the component $B^{0,2}$ must vanish and $\bar{\partial} B^{1,1}=0$. By the $\bar{\partial}$-Poincar\'e lemma, in some open subset $V\subset U$, we have $B^{1,1}=\bar{\partial} \tau$ for a $(1,0)$-form $\tau$. The form $$\tilde{B}=B^{2,0} - \partial \tau $$ is a $(2,0)$-form such that $d\tilde B=dB$ and then $(g,\tilde B)\in \cS(V)$ maps to $g\in \cO_G(V)$.

\end{proof}

The sequence \eqref{eq:sescE} induces a long exact sequence (of pointed sets) in cohomology
\begin{equation}\label{eq:lescEindcx}
  \xymatrix@R-2pc{
0 \ar[r] & H^0(\Omega^{2,0}_{cl}) \ar[r] & H^0(\cS) \ar[r] & H^0(\cO_{G}) \ar[r]^{\quad \delta_1} & \\
\ar[r] & H^1(\Omega^{2,0}_{cl}) \ar[r] & H^1(\cS) \ar[r] & H^1(\cO_{G}) \ar[r]^{\delta_2} & H^2(\Omega^{2,0}_{cl}),
}
\end{equation}
which we describe now explicitly.

We start by characterizing the cohomology of the sheaf $\Omega^{2,0}_{cl}$. Note that the abelian group $H^1(\Omega^{2,0}_{cl})$ parametrizes isomorphism classes of exact holomorphic Courant algebroids on the complex manifold $X$ \cite{G2}. Consider the differential graded sheaf $(\Omega^{\leqslant\bullet},d)$, defined by restriction of \eqref{eq:cxleqk}. By the existence of partitions of unity, this sheaf is moreover fine and hence soft. We use it to describe the cohomology.

\begin{lemma}\label{lemma:omega20k}
The sequence $0 \to \Omega^{2,0}_{cl} \xrightarrow{id} \Omega^{\leqslant \bullet}$ is a soft resolution of the sheaf  $\Omega^{2,0}_{cl}$ and hence
\begin{equation}\label{eq:cohomologyisom}
H^k(\Omega^{2,0}_{cl}) \cong H^k(\Omega^{\leqslant \bullet}).
\end{equation}
\end{lemma}

The $k$-th cohomology group $H^k(\Omega^{\leqslant\bullet})$ is explicitly given by
\begin{equation}\label{eq:kcohomologyleqk}
H^k(\Omega^{\leqslant\bullet}) = \frac{\Ker \; d \colon \Omega^{\leqslant k} \to \Omega^{\leqslant k+1}}{ \Im \; d \colon \Omega^{\leqslant k-1} \to\Omega^{\leqslant k}}.
\end{equation}

Next, we define a subgroup $I$ of the additive group $H^1(\Omega^{2,0}_{cl})$, induced by $\delta_1$ in \eqref{eq:lescEindcx}. Note that $H^0(\cO_G)$, given by global holomorphic maps from $X$ to $G$, has a natural group structure induced by pointwise multiplication and can be identified with the gauge group of the trivial bundle $X \times G$. We use the homomorphism $\sigma_P \colon \cG_P \to H^1(\Omega^{2,0}_{cl})$ in Lemma \ref{lemma:sigmaP}, for the case $P = X \times G$.

\begin{lemma}\label{lemma:delta1}
Via the isomorphism $H^1(\Omega^{\leqslant\bullet}) \cong H^1(\Omega^{2,0}_{cl})$, the map $\delta_1$ in \eqref{eq:lescEindcx} is given by
$$
\delta_1(g) = \sigma_{X \times G}(g) = [g^*\sigma_3],
$$ 
for $g \in H^0(\cO_G)$ and $\sigma_3$ the holomorphic Cartan $(3,0)$-form on $G$ associated to $c$. Consequently, $\delta_1$ is a morphism of groups, and we define
$$
I := \operatorname{Im} \delta_1 \subset H^1(\Omega^{2,0}_{cl}).
$$
\end{lemma}

\begin{proof}
The identification $\delta_1 = \sigma_{X \times G}$ as well as the explicit formula for $\delta_1$ follows from Lemma \ref{lemma:sigmaP} and equation \eqref{eq:dBCSexphol}, by a simple diagram chasing using Lemma \ref{lemma:exact-sequence-holomorphic}. 
\end{proof}

Recall that $H^1(\cO_{G})$ parameterizes isomorphism classes of holomorphic principal $G$-bundles on $X$. 
Given $[P] \in H^1(\cO_{G})$, it defines a complex first Pontryagin class
$$
p_1^c([P]) \in H^4(X,\CC),
$$
represented by $c(F_\theta \wedge F_\theta) \in \Omega^4$ for any choice of connection $\theta$ on the underlying smooth bundle $\underline{P}$. Denote by $\mathcal{A}_{P}$ the space of connections on $\underline{P}$ such that
\begin{equation}\label{eq:def-PA}
F_\theta^{0,2} = 0, \qquad P_\theta = P.
\end{equation}
As before, we denote by $P_\theta$ the holomorphic principal bundle induced by $\theta$. For $\theta$ in $\mathcal{A}_P$, we have
$$
c(F_\theta \wedge F_\theta) \in \Omega^{\leqslant 2}.
$$
Given another connection $\theta'$ with the same properties, setting  $a = \theta' - \theta \in \Omega^{1,0}(\ad \underline{P})$ there is an equality
$$
c(F_{\theta'} \wedge F_{\theta'}) - c(F_\theta \wedge F_\theta) = d \Big{(}2c(a \wedge F_\theta) + c(a \wedge d^\theta a) + \frac{1}{3}c(a,[a,a])\Big{)}
$$
and, by \eqref{eq:kcohomologyleqk},  we obtain a well-defined class
\begin{equation}\label{eq:p1c-H2leq}
p_1^c([P]) \in H^2(\Omega^{\leqslant \bullet}).
\end{equation}

\begin{proposition}\label{prop:lescEindcx}
There is an exact sequence of pointed sets
\begin{equation}\label{eq:lescEind2cx}
  \xymatrix{
0 \ar[r] & H^1(\Omega^{2,0}_{cl})/I \ar[r]^{\quad \iota} & H^1(\cS) \ar[r]^\jmath & H^1(\cO_{G}) \ar[r]^{^{p_1^c}\quad } & H^2(\Omega^{\leqslant \bullet}).
  }
\end{equation}
Furthermore, $\iota$ induces a transitive action of the additive group $H^1(\Omega^{2,0}_{cl})$ on the fibres of the map $H^1(\cS)\xrightarrow{\jmath} H^1(\cO_G)$.
\end{proposition}

\begin{proof}
The exactness on $H^1(\cS)$ follows from the long exact sequence \eqref{eq:lescEindcx} and the definition of $I$ in Lemma \ref{lemma:delta1}.

To prove the exactness on $H^1(\cO_G)$ we just need to prove that $\delta_2$ in \eqref{eq:lescEindcx} coincides with $p_1^c$ as in \eqref{eq:p1c-H2leq}, when we use the isomorphism \eqref{eq:cohomologyisom}. In order to spell this out, one uses the \v Cech cohomology isomorphism $H^2(\Omega^{\leqslant 0}_{cl})\cong H^1(\Omega^{\leqslant 1}_{cl})$, coming from the exact sequence
$$ 0 \to \Omega^{\leqslant 0}_{cl} \to \Omega^{\leqslant 0} \xra{d} \Omega^{\leqslant 1}_{cl} \to 0,$$
and then take a representative in $H^1(\Omega^{\leqslant 1}_{cl})$ to define an element in $\Omega^{\leqslant 2}_{cl}$, whose class gives the element in $H^2(\Omega^{\leqslant \bullet})$.

 Consider a good cover $\{U_i\}$ and denote the multiple intersections by $U_{ij}$, $U_{ijk}$, etc.  Given a class $[\{g_{ij}\}]\in H^1(\cO_G)$, we consider a cochain $\{ (g_{ij},B_{ij}) \}$ in $\cS$, with $B_{ij} \in \Omega^{2,0}$. Its coboundary is given by $\{ (1,d_{ijk}) \}$ for
\begin{equation}\label{eq:dijk}
d_{ijk}=B_{ij}+B_{jk}+B_{ki}+c(g_{kj}a_{ij}\wedge a_{jk}),
\end{equation} 
and the connecting homomorphism is
$\delta_2( [\{g_{ij}\}] ) = [ \{d_{ijk}\} ] \in H^2(\Omega^{2,0}_{cl})$. Here, we define $a_{ij}:= g_{ji}\theta_0-\theta_0$.

We want to write $d_{ijk}$ as a coboundary. For this, we choose a connection $\theta\in \mathcal{A}_P$, given on $U_i$ by $\theta_i$, so we have $g_{ij}\theta_j=\theta_i$. Set $a_i=\theta_i-\theta_0$, where $\theta_0$ denotes the trivial connection on $U_i$. We then have
\begin{equation}\label{eq:aij}
a_{ij}=g_{ji}\theta_0-\theta_0=a_j - g_{ji}a_i
\end{equation}
and the identity $g_{ij}a_{ij}=-a_{ji}$, which, together with the invariance of $c$, gives
\begin{align*}
c(g_{kj}a_{ij}\wedge a_{jk}) & {} = c(g_{kj}a_j\wedge a_{jk}) -c(g_{ki}a_i \wedge a_{jk} ) \\
& {} = - c(a_j\wedge a_{kj}) - c(a_i \wedge g_{ik}a_{jk}) \\
& {} = - c(a_j\wedge a_{kj}) - c(a_i \wedge g_{ik}a_{k}) + c(a_i \wedge g_{ij}a_{j}) \\
& {} = - c(a_j\wedge a_{kj}) -c(g_{ki}a_i \wedge a_{k}) + c(a_i \wedge g_{ij}a_{j}) \\
& {} = - c(a_j\wedge a_{kj}) -c((a_k-a_{ik}) \wedge a_{k})) + c(a_i \wedge (a_i-a_{ji}))\\
& {} = - c(a_j\wedge a_{kj}) -c(a_k \wedge a_{ik}) - c(a_i \wedge a_{ji}).
\end{align*}

From this and \eqref{eq:dijk} we have
\begin{equation*}\label{eq:dijk-as-coboundary}
\{ d_{ijk} \} = \delta( \{B_{ij} - c(a_i\wedge a_{ji}) \})
\end{equation*}
and $[\{d_{ijk}\}]$ equals $ [ \{ d( B_{ij} - c(a_i\wedge a_{ji} ) ) \} ]$ in the \v Cech cohomology $H^1(\Omega^{\leqslant 1}_{cl}). $
Again, we want to express the representative as a coboundary. Using \eqref{eq:dBCShol},
\begin{equation*}\label{eq:d-as-CS-difference}
d( B_{ij} - c(a_i\wedge a_{ji} ) ) = CS(g_{ji}\theta_0) - CS(\theta_0) -dc(a_{ij}\wedge \theta_0) - dc(a_i\wedge a_{ji}).
\end{equation*}
We add and substract terms in this expression to get
\begin{align}\label{eq:diff-CS}
& (CS (\theta_j) - CS(\theta_0) - dc(\theta_j\wedge \theta_0) ) \nonumber \\
& - (CS(g_{ji}\theta_i) - CS(g_{ji}\theta_0) -dc(g_{ji}\theta_i\wedge g_{ji}\theta_0))   \\
& - dc(g_{ji} \theta_i \wedge g_{ji}\theta_0) + dc(\theta_j\wedge \theta_0) - dc(a_{ij}\wedge \theta_0) - dc(a_i\wedge a_{ji}). \nonumber
\end{align}
The third line is identically zero as the first two summands equal, by \eqref{eq:aij},
$$-dc(\theta_j\wedge g_{ji}\theta_0) + dc(\theta_j\wedge \theta_0) = -dc(\theta_j \wedge (g_{ji}\theta_0-\theta_0) ) = -dc(\theta_j \wedge a_{ij}),$$
whereas the last two summands equal, by \eqref{eq:aij} again,
$$ - dc(a_{ij}\wedge \theta_0) + dc(g_{ji}a_i\wedge a_{ij}) = dc((\theta_0+a_j-a_{ij})\wedge a_{ij})=dc(\theta_j \wedge a_{ij}).$$
The second line of \eqref{eq:diff-CS} is a basic form, so we can pull it back by $g_{ij}$ and obtain the coboundary
\begin{equation}\label{eq:dB-caa-as-coboundary}
d(B_{ij}-c(a_i\wedge a_{ji}))=\delta( CS (\theta_i) - CS(\theta_0) - dc(\theta_i\wedge \theta_0) ) .
\end{equation}
Consequently, $\delta_2( [\{g_{ij}\}] )$ is represented by
$$ d(CS (\theta_i) - CS(\theta_0) - dc(\theta_i\wedge \theta_0) ),$$
which gives a class in $H^2(\Omega^{\leqslant \bullet})$. By \eqref{eq:diff-of-CS}, this equals the representative of the Pontryagin class $ c(F_{\theta_i}\wedge F_{\theta_i}),$ that is, for $P$  the holomorphic bundle corresponding to $\{g_{ij}\}$, we have $$\delta_2( [\{g_{ij}\}] ) = p_1^c([P]).$$

Finally, as for the $H^1(\Omega^{2,0}_{cl})$-action, an element $[\{(0,C_{ij})\}]\in H^1(\Omega^{2,0}_{cl})$ acts  on $[\{(g_{ij},B_{ij})\}]\in H^1(\cS)$ by $[\{(g_{ij},B_{ij}+C_{ij})\}]$. The transitivity of the action is straightforward, as if $[\{(g_{ij},B_{ij})\}]$ and $[\{(g_{ij}',B_{ij}')\}]$ are in $j^{-1}([\{g_{ij}\}])$ then $[\{(g_{ij}',B_{ij}')\}] = [\{(g_{ij},B_{ij}'')\}]$ for suitable $B_{ij}'' \in \Omega^{2,0}(U_{ij})$, and therefore $B_{ij}'' - B_{ij}$ defines a $1$-cocycle for $\Omega^{2,0}_{cl}$.
\end{proof}

The isotropy of the $H^1(\Omega^{2,0}_{cl})$-action on $j^{-1}([P])$ can be described explicitly in terms of the gauge group of $P$. We shall postpone the proof of this fact until Proposition \ref{theo:deRhamC}.

In our next result, which follows as a straightforward consequence of Proposition \ref{prop:lescEindcx}, we show that for $\partial\dbar$-manifolds the existence of a Courant extension \eqref{eq:defstring} on a given holomorphic principal bundle reduces to the topological condition \eqref{eq:p1dR}.

\begin{corollary}\label{cor:ddbar}
Assume that $X$ is a $\partial\dbar$-manifold. Then, the natural map in cohomology
\begin{equation}\label{eq:cohomomap}
H^2(\Omega^{\leqslant \bullet}) \to H^4(X,\CC)
\end{equation}
is injective. Consequently, a holomorphic principal bundle $P$ admits a Courant extension \eqref{eq:defstring} if and only if
\begin{equation}\label{eq:p1dRbis}
p^c_1(P) = 0 \in H^4(X,\CC).
\end{equation}
\end{corollary}

\begin{proof}
Assume that $\tau \in \Omega^{\leqslant 2}$ satisfies $\tau = d h$ for some $h \in \Omega^3_\CC$. Then, by type decomposition it follows that
$$
\tau = d (h^{3,0} + h^{2,1}) + \partial h^{1,2}, \qquad \dbar h^{1,2} = 0.
$$
Applying the $\partial\dbar$-Lemma, there exists $b \in \Omega^{1,1}$ such that $\partial h^{1,2} = \partial \dbar b^{1,1}$ and therefore
$$
\tau = d (h^{3,0} + h^{2,1}) + \partial \dbar b^{1,1} = d(h^{3,0} + h^{2,1} - \partial b^{1,1}).
$$
We conclude that \eqref{eq:cohomomap} is injective. Finally, if \eqref{eq:p1dRbis} holds and $\theta$ is a connection on $\underline{P}$ such that $F_\theta^{0,2}= 0$ and $P_\theta = P$, by the injectivity of \eqref{eq:cohomomap} it follows that $dh + c(F_\theta \wedge F_\theta) = 0$ for some $h \in \Omega^{\leqslant 1}$. The `if part' in the second part of the statement follows now from Proposition \ref{prop:strholCour}. The `only if part' is trivial.
\end{proof}

To finish this section, we analyze a particular case of Proposition \ref{prop:lescEindcx} which is relevant for the applications to the theory of metrics on holomorphic Courant algebroids developed in \cite{grt2}.

\begin{corollary}\label{lemma:delta10hol}
Assume that $G$ is a reductive complex Lie group and that $X$ is compact. Then $H^0(\cO_{G}) = G$ and $\delta_1 = 0$. Consequently, there is an exact sequence of pointed sets
\begin{equation*}\label{eq:lescEind2cxdelta0}
  \xymatrix{
0 \ar[r] & H^1(\Omega^{\leqslant \bullet}) \ar[r]^{\quad \iota} & H^1(\cS) \ar[r]^\jmath & H^1(\cO_{G}) \ar[r]^{^{p_1^c}\quad } & H^2(\Omega^{\leqslant \bullet}).
  }
\end{equation*}
\end{corollary}

\begin{proof}
Elements of $H^0(\cO_{G})$ are global holomorphic maps $g \colon X \to G$. Since $G$ is reductive, it admits a holomorphic embedding into $\CC^k$ for sufficiently large $k$, and therefore $g$ must be constant by the maximum principle. Consequently, $g^*\sigma_3=0$, and we conclude $\delta_1 = 0$ by Lemma \ref{lemma:delta1}. The proof follows from Proposition \ref{prop:lescEindcx}.
\end{proof}



\subsection{Classification and holomorphic string classes}\label{subsec:holstringclass}

To start this section, we obtain our classification of string algebroids in terms of the first \v Cech cohomology of the sheaf $\cS$. We will use the description of the set of isomorphism classes in Proposition \ref{prop:classification}.

\begin{theorem}\label{th:classification}
There is a one to one correspondence between elements in $H^1(\cS)$ and isomorphism classes of string algebroids.
\end{theorem}


\begin{proof}
First, we give a description of $H^1(\cS)$ using the $1$-cocycles $Z^1(\cO_G)$ on a good cover, whose elements $\{g_{ij}\}$ are identified with the holomorphic principal bundle $P$ which they glue to. By Proposition \ref{prop:classification}, the statement reduces to prove that $H^1(\cS)$ is bijective to 
$$\{ (P,H,\theta) \st P\in Z^1(\cO_{G}), H\in \Omega^{\leqslant 1}, \theta\in \cA_{P}, dH= - c(F_\theta\wedge F_\theta) \} / \sim, $$
where $\sim$ is the equivalence relation generated by $(P,H,\theta)\sim (P,H',\theta')$ when 
\begin{equation}\label{eq:equiv-rel-1}
H' = H + CS(\theta) - CS(\theta') - dc(\theta \wedge \theta') + dB
\end{equation} for some $B\in \Omega^{\leqslant 0}$, and, for any $0$-cochain $\{h_i\}$, which we write $\{h_i\}\in C^0(\cO_{G})$,
\begin{equation}\label{eq:equiv-rel-2}
(\{g_{ij}\}, H, \{\theta_j\})\sim (\{h_i g_{ij} h_j^{-1} \}, H, \{h_i\theta_i \}),
\end{equation}
where the principal bundle can change. 
Take a representative $(g_{ij},B_{ij})$ of a class in $H^1(\cS)$. To define a triple $(P,H,\theta)$, we set $P=\{g_{ij}\}$, choose any connection $\theta\in \cA_{P}$ on $P$, locally given by $\theta_j$, and use the proof of Proposition \ref{prop:lescEindcx} to define $H$. By the cocycle condition,
$$(1,0)=\delta (g_{ij},B_{ij})  = (\delta g_{ij}, \delta(B_{ij}-c(a_i\wedge a_{ji}))),$$
As $\Omega^{2,0}$ is an acyclic sheaf, i.e., $H^1(\Omega^{2,0})=0$, there exist $\{C_i\}$, not necessarily closed $(2,0)$-forms defined up to addition of a global $(2,0)$-form $B$, such that
\begin{equation}\label{eq:Bij-as-Cj-Ci}
B_{ij}-c(a_i\wedge a_{ji})=C_j-C_i.
\end{equation}
By differentiating and using \eqref{eq:dB-caa-as-coboundary}, we have that the expression
\begin{equation}\label{eq:def-Hi}
H_i:=dC_i - CS(\theta_i) + CS(\theta_0) + dc(\theta_i\wedge\theta_0)
\end{equation}
defines a global form $H\in \Omega^{\leqslant 1}$, up to addition of an exact form $dB$. Just as in the proof of Proposition \ref{prop:lescEindcx} we get
$$ dH+c(F_\theta\wedge F_\theta)=0. $$

If we choose a different connection $\theta'\in \cA_{P}$ locally given by $\theta'_i$, we define $a_i'=\theta_i'-\theta_0$. When repeating the argument above we obtain, by \eqref{eq:Bij-as-Cj-Ci},
$$ C_j' - C_i'= B_{ij}-c(a_i'\wedge a_{ji})=B_{ij}-c(a_i\wedge a_{ji}) + c( (a_i-a_i')\wedge a_{ji}).$$
By using \eqref{eq:aij} and $g_{ji}(a_i-a_i')=g_{ji}(\theta_i-\theta_i') = \theta_j-\theta_j'=a_j-a_j'$, we get
$$c( (a_i-a_i')\wedge a_{ji})= c( (a_i-a_i')\wedge (a_i - g_{ij} a_j)) = -c(a_i' \wedge a_i) + c(a_j' \wedge a_j),$$
that is, $C_i'=C_i + c(a_i'\wedge a_i)$, again, up to addition of a global $(2,0)$-form $B$. When using \eqref{eq:def-Hi} to compute $H_j'$ and compare it to $H_j$, we obtain
\begin{align*}
H_i' & {} =dC_i'-CS(\theta'_i)+CS(\theta_0)+dc(\theta_i'\wedge \theta_0) \nonumber\\
& {}= dC_i - CS(\theta_i) + CS(\theta_0)+dc(\theta_i\wedge \theta_0) \\
& \phantom{ {} = } + CS(\theta_i) - CS(\theta_i')  + dc( a_i'\wedge a_i ) + dc((\theta'_i-\theta_i)\wedge \theta_0 )\nonumber\\
& {}= H_i + CS(\theta_i) - CS(\theta_i') - dc(\theta_i \wedge \theta_i' ), \nonumber
\end{align*}
up to addition of $dB$, where the last equality follows from
$$ dc( a'_i \wedge a_i ) = dc( (\theta_i'-\theta_0) \wedge (\theta_i-\theta_0) ) = dc(\theta_i'\wedge \theta_i') - dc((\theta'_i-\theta_i)\wedge \theta_0 ). $$
Thus, $(P,H,\theta)\sim (P,H',\theta')$ for the relation in the statement of the theorem.



We check now that the definition above does not depend on the choice of representative $(g_{ij},B_{ij})$. Consider a different representative given, for some $1$-cochain $\{(h_i,B_i)\}$, by
\begin{align*}
(\tilde{g}_{ij},\tilde{B}_{ij}) & {} =(h_i,B_i)(g_{ij},B_{ij})(h_j,B_j)^{-1}\\
& {} = (h_i g_{ij} h_j^{-1}, B_{ij}+B_i-B_j + c(g_{ji}a^{h_i}\wedge a_{{ij}}) + c(h_j a^{h_i g_{ij}}\wedge a^{h_j^{-1}})). \nonumber
\end{align*}
By using the invariance of $c$ and \eqref{eq:afg-af-ag}, we have
\begin{align}
\tilde{B}_{ij} & {} = B_{ij}+B_i-B_j - c(a^{h_i}\wedge a_{{ji}}) - c( (g_{ji}a^{h_i} + a_{ij})\wedge a^{h_j}). \nonumber
\end{align}
In order to  get the corresponding $[(P,\tilde{H},\tilde{\theta})]$, we choose, for convenience in the calculations below, the connection given by $\tilde{\theta}_j=h_j\theta_j$,  which yields the identity
\begin{align}\label{eq:tilde-a}
\tilde{a}_i & {} = h_i\theta_i- \theta_0 = h_i a_i + a^{h_i^{-1}}= h_i(a_i - a^{h_i}),
\end{align}
and consider the expression
\begin{equation}\label{eq:tildeC}
\tilde{B}_{ij}-c(\tilde{a}_i\wedge \tilde{a}_{ji}) = \tilde{C}_j-\tilde{C}_i.
\end{equation}
We relate it to $H_j$ by using \eqref{eq:Bij-as-Cj-Ci}. The left-hand side of \eqref{eq:tildeC} equals
$$ C_j  - C_i + B_i - B_j + c(a_i\wedge a_{ji}) - c(a^{h_i}\wedge a_{{ji}}) - c( (g_{ji}a^{h_i} + a_{ij})\wedge a^{h_j}) - c(\tilde{a}_i\wedge \tilde{a}_{ji}).$$
We deal now with the terms with $c$. The last one, by \eqref{eq:afg-af-ag} and \eqref{eq:tilde-a}, is
$$  -c(\tilde{a}_i\wedge a^{h_j g_{ji} h_i^{-1}}) =  -c(\tilde{a}_i\wedge (h_i g_{ij}a^{h_j} + h_i a_{ji} + a^{h_i^{-1}})),$$
while the first two,  by \eqref{eq:tilde-a}, are
$$ c((a_i - a^{h_i})\wedge a_{{ji}}) = c( h_i^{-1}\tilde{a}_i \wedge a_{ji} ) = c(\tilde{a}_i \wedge h_i a_{ji}).$$
There is a cancellation and we have that all the terms with $c$ are
\begin{equation*}\label{eq:intermediate-step}
-c(\tilde{a}_i\wedge (h_i g_{ij}a^{h_j} + a^{h_i^{-1}})) - c( (g_{ji}a^{h_i} + a_{ij})\wedge a^{h_j}).
\end{equation*}
On the one hand,
$$ -c(\tilde{a}_i\wedge a^{h_i^{-1}})= c(h_i^{-1}\tilde{a}_i \wedge a^{h_i}) = -c(a_i\wedge a^{h_i}),$$
whereas the rest of the terms are
$$ -c((g_{ji}h_i^{-1} \tilde{a}_i - g_{ji}a^{h_i} - a_{ij})\wedge a^{h_j})=c((g_{ji}a_i - a_{ij}+g)\wedge a^{h_j})=c(a_j\wedge a^{h_j}).$$
So \eqref{eq:tildeC} becomes, up to addition of a global two-form,
$$\tilde{C}_i=C_i - B_i + c(a_i\wedge a^{h_i}).$$
The corresponding $\tilde{H}$ is, by \eqref{eq:dBCShol} and using that the pullback by $h_i$ preserves basic forms,
\begin{align*}
\tilde{H}_i & {} =d\tilde{C}_i - CS(\tilde{\theta}_i)+ CS(\theta_0) + dc(\tilde{\theta}_i\wedge \theta_0) \\
& {} = dC_i - (CS(h_i^{-1}\theta_0) - CS(\theta_0) -dc(h_i^{-1}\theta_0\wedge \theta_0)) -dc(a_i\wedge a^{h_i})\\
&\phantom{ {} = } -h_i^{-1}( CS(\tilde{\theta}_i) - CS(\theta_0) -dc(\tilde\theta_i\wedge \theta_0) )\\
& {} = dC_i - CS(\theta_i)  +CS(\theta_0) + dc((\theta_i-\theta_0)\wedge h_i^{-1}\theta_0) + dc(a_i\wedge a^{h_i})\\
& {} =  dC_i - CS(\theta_i)  +CS(\theta_0) + dc(\theta_i\wedge \theta_0) + 2dc(a_i\wedge a^{h_i}),
\end{align*}
that is, $H_i$ up to addition of an exact form, as $c(a_i\wedge a^{h_i})\in \Omega^{\leqslant  0}$. Hence, by \eqref{eq:equiv-rel-2}, the image does not depend on the representative.

The calculations above also show that the map defined is injective. To show that it is surjective, consider a class $[(P,H,\theta)]$. We need to find an element $(g_{ij},B_{ij})$ mapping to it. The representative $P=\{g_{ij}\}$ determines an element of $Z^1(\cO_G)$. By $dH = - c(F_\theta\wedge F_\theta)$, we can find $\{C_i\}$ such that \eqref{eq:def-Hi} is satisfied. Define then $B_{ij}$ by \eqref{eq:Bij-as-Cj-Ci}, where $a_j$ is determined by $\theta$ and $a_{ij}$ by $P$.
\end{proof}

\begin{remark}
	A much simpler but inspirational setting for part of the technicalities in the proofs of Proposition \ref{prop:lescEindcx} and Theorem \ref{th:classification} is \cite[Sec. 2.3.1]{Rubioth} and \cite{Rubio}.
\end{remark}

Let $P$ be a holomorphic principal $G$-bundle over $X$. Building on Proposition \ref{prop:classification}, Proposition \ref{lemma:delta10hol}, and Theorem \ref{th:classification}, we provide next a characterization of $\jmath^{-1}([P])$ \emph{\`a la} de Rham and describe the isotropy of the corresponding $H^1(\Omega^{2,0}_{cl})$-action (see Proposition \ref{prop:lescEindcx}). As mentioned in Section \ref{sec:intro} (see also Appendix \ref{sec:smooth}), up to isotropy the elements in $\jmath^{-1}([P])$ provide analogues in the holomorphic category of Redden's real string classes \cite{Redden}, and hence we will refer to them as \emph{holomorphic string classes}.

As before, we will denote by $\cG_P$ the gauge group of $P$. 
Relying on Lemma \ref{lemma:sigmaP} and Lemma \ref{lemma:omega20k} we can give the following definition.

\begin{definition}
We define the additive subgroup
$$
I_{[P]} :=  \Im \; \sigma_P = \langle \sigma_P(g) \; | \; g \in \cG_P \rangle \subset H^1(\Omega^{2,0}_{cl}).
$$
\end{definition}

It is easy to see that $I_{[P]}$ is an invariant of the isomorphism class $[P]$, which justifies the notation. We can now state our characterizing \emph{holomorphic string classes} for the principal bundle $P$.

\begin{proposition}\label{theo:deRhamC}
There is a natural bijection
\begin{equation*}\label{eq:lescEind3}
\jmath^{-1}([P]) \cong \{(H,\theta) \in \Omega^{\leqslant 1} \times \cA_{P} \; | \; dH + c(F_\theta \wedge F_\theta) = 0\}/ \sim,
\end{equation*}
where $(H,\theta)\sim (H',\theta')$ if, for some $B\in \Omega^{\leqslant 0}$ and $g\in \cG_P$ 
\begin{equation}\label{eq:anomaly}
H'  = H + CS(g\theta) - CS(\theta') - dc(g\theta \wedge \theta') + dB.
\end{equation}
Furthermore, the isotropy of $\jmath^{-1}([P])$ for the transitive $H^1(\Omega^{2,0}_{cl})$-action is $I_{[P]}$.
\end{proposition}

\begin{proof}
To describe the fibre over $[P]$, 
we have to consider the equivalence relation \eqref{eq:equiv-rel-1} combined with the equivalence relation \eqref{eq:equiv-rel-2} for those $\{h_i\}\in C^0(\cO_{G})$ such that $g_{ij}=h_ig_{ij}h_j^{-1}$. These  $\{h_i\}\in C^0(\cO_{G})$ define precisely the automorphisms of $P$ and the isomorphism of the statement follows.


Finally, to describe the isotropy of the $H^1(\Omega^{2,0}_{cl})$-action, we simply notice that if $[(H,\theta)] = [(H + C,\theta)]$ for $[C] \in H^1(\Omega^{2,0}_{cl})$, then \eqref{eq:anomaly} implies that $[C] = [\sigma(g)]$ for some $g \in \cG$, as claimed.

\end{proof}

\section{Deformation theory}
\label{sec:deformation-theory}

In this section we study the deformation theory for string algebroids over a fixed compact complex manifold $X$.

\subsection{Deformations of string algebroids}

Let $(Q,P,\rho)$ be a string algebroid over $X$, with structure group $G$. By a result of Ehresmann \cite{Ehresmann}, a complex deformation $P_t$ of the holomorphic bundle $P$ can be regarded as a smooth family of integrable $(0,1)$-connections $\theta_t^{0,1}$ on the smooth $G$-bundle $\underline{P}$. Relying on Proposition \ref{prop:classification}, it is natural to give the following definition.

\begin{definition}
\label{def:deformation Q}
 A smooth (respectively analytic) string algebroid deformation of $(Q,P,\rho)$ is a family of string algebroids $(Q_t,P_t,\rho_t)_{t\in \Lambda}$ over the complex manifold $X$, with structure group $G$, parametrized by $\Lambda$, satisfying:
 \begin{itemize}
  \item[i)] $(Q_0,P_0,\rho_0)=(Q,P,\rho)$,
  \item[ii)] the connexion $\theta_t^{0,1}$ is smooth (respectively analytic) with respect to $t$,
  \item[iii)] for each $t$ there exists a representative $(H_t, \theta_t) \in \Omega^{\leqslant 1} \times \cA_{P_t}$ of the isomorphism class of $(Q_t,P_t,\rho_t)$, as in Proposition \ref{theo:deRhamC}, such that $(H_t, \theta_t)$ is smooth (respectively analytic) in $t$.
\end{itemize}
\end{definition}

In the previous definition we consider $\Omega^{\leqslant 1} \times \cA_{P_t} \subset \Omega^3_\CC \times \cA_{\underline{P}}$, where $\cA_{\underline{P}}$ denotes the space of connections on $\underline{P}$. 
When the underlying principal bundle $P$ is fixed, a string algebroid $(Q,P,\rho)$ can always be deformed by the action of the additive group $H^1(\Omega^{2,0}_{cl})/I_{[P]}$ (see Proposition \ref{prop:lescEindcx}). The main question we address here is: given a deformation $P_t$ of the underlying principal bundle $P$, can we find a string algebroid deformation of $(Q,P,\rho)$ of the form $(Q_t,P_t,\rho_t)$? Our first result provides an affirmative answer to this question for the case of $\partial \dbar$-manifolds and complex reductive Lie group $G$. The proof relies on Corollary \ref{cor:ddbar}.

\begin{proposition}\label{prop:Stringdeformation}
Let $X$ be a compact complex manifold satisfying the $\partial\dbar$-Lemma and $G$ be a reductive complex Lie group. Let $(Q,P,\rho)$ be a string algebroid over $X$ with structure $G$. Let $(P_t)_{t\in \Lambda}$ a deformation of the underlying holomorphic principal bundle. Then, there exists a string algebroid deformation $(Q_t,P_t,\rho_t)_{t\in \Lambda}$ with underlying principal bundle $P_t$ for all $t$.
\end{proposition}

\begin{proof}
Let $(H,\theta) \in \Omega^{\leqslant 1} \times \cA_{P}$ be a pair representing the holomorphic string class of $(Q,P,\rho)$ (see Proposition \ref{theo:deRhamC}). We fix a reduction $h$ of $\underline{P}$ to a maximal compact subgroup, and denote by $\theta_{h,t}$ the Chern connection of $h$ in the holomorphic principal bundle $P_t$.  Then, $(\tilde H,\theta_{h,0}) \sim (H,\theta)$, where
\begin{equation*}\label{eq:lem:BCdeformation}
\tilde H = H + CS(\theta) - CS(\theta_{h,0}) - dc(\theta \wedge \theta_{h,0}).
\end{equation*}
It suffices to show that there exists $H_t = \tilde H + \tau_t \in \Omega^{\leqslant 1}$ varying smoothly with $t$ such that
\begin{equation}\label{eq:dtaut}
d\tau_t + c(F_{\theta_{h,t}} \wedge F_{\theta_{h,t}}) - c(F_{\theta_{h,0}}\wedge F_{\theta_{h,0}}) = 0.
\end{equation}
We first argue that there exists one such $\tau_t$ for each $t$, and then prove the smooth dependence on the parameter. To prove this, notice that,
$$
c(F_{\theta_{h,t}} \wedge F_{\theta_{h,t}}) - c(F_{\theta_{h,0}}\wedge F_{\theta_{h,0}}) = dT_t
$$
with
$$
T_t = 2c(a_t \wedge F_{\theta_{h,0}}) + c(a_t \wedge d^{\theta_{h,0}} a_t) + \frac{1}{3}c(a_t,[a_t,a_t]),
$$
and $a_t =  \theta_{h,t} - \theta_{h,0}  = \Omega^{1}_\CC(\ad \underline{P})$. Since $dT_t$ is of type $(2,2)$, furthermore we have
$$
d T_t = \partial T_t^{1,2} + \dbar T_t^{2,1}.
$$
We define
$$
\tilde \tau_t = - T^{3,0 + 2,1 + 1,2}_t,
$$
so that
$$
d \tilde \tau_t + dT_t = d T^{0,3}_t.
$$
Since $\dbar T^{0,3}_t = 0$, by the $\partial\dbar$-Lemma there exists $b_t \in \Omega^{0,2}$ such that
$$
\partial T^{0,3}_t = \partial \dbar b_t.
$$
Setting now $\hat \tau_t = \tilde \tau_t + \partial b_t$, we obtain $d \hat \tau_t + dT_t = \dbar \partial b^{0,2}_t + \partial T^{0,3}_t = 0$. Using again that $dT_t \in \Omega^{2,2}$, it follows that $\dbar \hat \tau^{1,2}_t = 0$. By the $\partial\dbar$-Lemma there exists $\chi_t \in \Omega^{1,1}$, such that
$$
\partial \hat \tau^{1,2}_t = \partial \dbar \chi_t.
$$
Setting $\tau_t = \hat \tau_t^{3,0} + \hat \tau_t^{2,1} - \partial \chi_t$, it follows that $\tau_t$ satisfies \eqref{eq:dtaut} as required.

The final step of the proof is to show that we can choose $b_t$ and $\chi_t$ varying smoothly with $t$. Considering a hermitian metric $g$ on $X$, we use the decomposition
$$
\Om^{1,3} = \cH^{1,3}_{BC}\oplus \Im (\del \delb) \oplus \Im(\del^* \oplus \delb^*).
$$
induced by the (elliptic) Bott-Chern laplacian (see \cite{Sch})
$$
\Delta^{BC} = \del\delb (\del\delb)^* + \dbar^* \partial^* \del\delb + \dbar^* \partial \del^* \delb + \del^* \delb \delb^* \del + \dbar^* \delb +  \partial^* \del. 
$$
Let $G$ be the corresponding Green operator and define
$$
x_t = G(\partial T^{0,3}_t),
$$
which varies smoothly in $t$ since $\partial T^{0,3}_t$ does. By the $\partial\dbar$-Lemma we have $\partial T^{0,3}_t \in \Im (\del \delb)$, and therefore we can assume that $x_t \in \Im (\del \delb)$. Consequently, we have $\del x_t = \delb x_t=0$ and thus 
$\del\delb (\del\delb)^* x_t = \Delta^{BC} x_t$. Setting now $b_t = (\del\delb)^* x_t$, which is also smooth in $t$, we have
$$
\partial \dbar b_t = \Delta^{BC} \circ G(\partial T^{0,3}_t) = \partial T^{0,3}_t,
$$
as required. Finaly, the proof for $\chi_t$ is analogous.
\end{proof}

In general, when the complex manifold $X$ does not satisfy the $\partial\dbar$-Lemma, we cannot expect the statement of Proposition \ref{prop:Stringdeformation} to hold. Our next goal is to construct a first-order obstruction to deform a string algebroid along with the underlying principal bundle.

\subsection{Sheaf resolutions and a first-order obstruction}

Fix a string algebroid $(Q,P,\rho)$ over $X$ with structure group $G$. The aim of this section is to obtain a first-order obstruction to deform $(Q,P,\rho)$, given a deformation of $P$. We do this in a systematic way, by means of constructing a differential graded algebra ruling the deformation theory. Following Kodaira and Spencer \cite{KS}, such a differential graded algebra can be obtained from a fine resolution of the sheaf of automorphisms of $(Q,P,\rho)$. 

By Proposition \ref{theo:deRhamC} we can fix a representative $(H,\theta)\in j^{-1}([P])$ of the associated holomorphic string class. Denote by $F=F_\theta$ the curvature of $\theta$. Up to the equivalence relation $\sim$ defined by (\ref{eq:anomaly}), by choosing the Chern connection of a fixed hermitian structure on $P$, we can assume that $F^{2,0}=0$.
Let $\cG_P$ be the group of holomorphic gauge transformations of $P$, with Lie algebra
$$
\Lie \;\cG_P =\lbrace \alpha\in \Om^0(\ad P)\:\vert\: \delb^\theta \alpha=0 \rbrace.
$$
For simplicity, we will denote $(Q,P,\rho)$ just by $Q$.

From Lemma \ref{lemma:Autstring}, the group of automorphisms of $Q$ is given by the set of pairs
$ (g, B)\in \cG_P\times \Om^{2,0}$ such that
\begin{equation*}
d B = 2 c (a^g, F) + c( a^g, d^\theta a^g) +\frac{1}{3} c(a^g, [a^g,a^g]),
\end{equation*}
where we recall that $a^g:=g^{-1}\theta - \theta$, with group structure given by (\ref{eq:groupproduct}).  The sheaf of Lie algebras $\faut_Q$ associated to this group is then given, for an open set $U\subset X$, by
$$
\faut_Q(U)=\lbrace (\alpha,b)\in  \Om^0(U,\ad P) \times \Om^{2,0}(U)\: \vert\:\: \delb ^\theta \alpha=0,\: d(b-2c(\alpha F))=0 \rbrace.
$$
Note that, using $\delb ^\theta \alpha=0$ combined with the Bianchi identity, the second condition for the pair $(\alpha,b)$ can we written as $db-2c(\partial^{\theta}\alpha \wedge F)=0$.

Denote by $\Omega^{2,0}_{cl}$ the sheaf of closed $(2,0)$-forms and by $\faut_{P}$ the sheaf of holomorphic sections of $\ad P$.
From the proof of Lemma \ref{lemma:exact-sequence-holomorphic},
we deduce the following exact sequence of sheaves:
\begin{equation}
\label{seq:decomp sheaf Q to P}
0 \to  \Omega^{2,0}_{cl} \to \faut_Q \to \faut_{P} \to 0.
\end{equation}
\begin{corollary}
 The sheaf $\faut_Q$ is acyclic.
\end{corollary}
\begin{proof}
The sheaves $\faut_{P}$ and $\Omega^{2,0}_{cl}$ are acyclic by the Poincar\'e-Dolbeault Lemma.
For a given polydisc $U\subset X$, the short exact sequence (\ref{seq:decomp sheaf Q to P})
implies the short exact sequence of sheaves
$$
0 \to  (\Omega^{2,0}_{cl})_{\vert U} \to (\faut_Q)_{\vert U} \to (\faut_{P})_{\vert U} \to 0.
$$
From the associated long exact sequence, and as $\faut_{P}$ and $\Omega^{2,0}_{cl}$ are acyclic,
we deduce that $\faut_Q$ is acyclic.
\end{proof}
Following \cite{KS}, the \v Cech cohomology of $\faut_Q$ describes deformations of $Q$.
We will build a fine resolution for this sheaf, using the fine resolutions
\begin{equation}
 \label{seq:resolution sheaf P}
 0\to \faut_{P} \lra{} \Om^{0,0}(\ad(P))\lra{\delb^\theta}\Om^{0,1}(\ad(P)) \lra{\delb^\theta}\ldots
\end{equation}
and
\begin{equation}
 \label{seq:resolution sheaf H}
 0\to \Omega^{2,0}_{cl} \lra{} \Om^{\leqslant 0}\lra{d}\Om^{\leqslant 1} \lra{d}\ldots
\end{equation}
We set
$$
\fL_Q^\bullet := \Om^{0,\bullet}(\ad P) \times \Om^{\leqslant \bullet},
$$
and denote by $\pi_1$, $\pi_2$ the projections onto the first and second factors, respectively. We introduce the  operator
\begin{equation*}
 \label{eq:differential Q}
 \begin{array}{cccc}
  d_Q : & \fL_{Q}^\bullet & \rightarrow & \fL_{Q}^{\bullet+1} \\
           &  (\alpha,b)           & \mapsto & (\delb^\theta \alpha,db-2c(\del^\theta \alpha\wedge F)).
 \end{array}
\end{equation*}
%
Note that $\faut_Q(X)=\ker d_Q \colon \fL_{Q}^0 \to \fL_{Q}^1$.
\begin{lemma}
\label{lem:differential Q}
 The operator $d_Q$ satisfies $d_Q\circ d_Q=0$.
\end{lemma}
\begin{proof}
 Let $(\alpha,b)\in\fL_Q^k$. As $\delb^\theta $ is a differential, the first component of $d_Q\circ d_Q(\alpha,b)$ vanishes.
 For the second component, using $dF=0$ and $F=F^{1,1}$,
  \begin{align*}
\pi_2(d_Q\circ d_Q(\alpha,b))  = {}& d ( db -2c(\del^\theta \alpha \wedge F) ) - 2 c(\del^\theta\delb^\theta \alpha  \wedge F)\\
                         = {}& -2c((d^\theta\del^\theta \alpha  + \del^\theta\delb^\theta \alpha ) \wedge F)\\
                         = {}& -2 c ([F,\alpha ]\wedge F),
\end{align*}
which vanishes by $\ad$-invariance of $c$ and $F$ being an $\ad P$-valued $2$-form.
\end{proof}

The complex $(\cL^\bullet_Q,d_Q)$ fits into an exact sequence of differential graded algebras:
\begin{equation}
\label{seq:decomp diff graded Q to P}
0 \to  (\Om^{\leqslant \bullet},d) \to (\cL^\bullet_Q,d_Q) \to (\Om^{0,\bullet}(\ad P),\dbar^\theta) \to 0,
\end{equation}
where the first map is the obvious inclusion and the second map is induced by the projection $\pi_1$. Using the short exact sequence (\ref{seq:decomp sheaf Q to P}), and the fine resolutions (\ref{seq:resolution sheaf P})  and  (\ref{seq:resolution sheaf H}), the following is straightforward:

\begin{lemma}
\label{lem:resolution sheaf Q}
 The complex
 \begin{equation*}
 \label{seq:resolution sheaf Q}
 0\to \faut_Q \lra{} \fL_Q^0 \lra{d_Q} \fL_Q^1 \lra{d_Q}\ldots
\end{equation*}
provides a fine resolution of $\faut_Q$. Consequently, the \v Cech cohomology of $\faut_Q$ is isomorphic to the cohomology of $(\cL^\bullet_Q, d_Q)$.
\end{lemma}

Our next result provides the desired obstruction to the deformation of string algebroids on a general complex manifold. Observe that the sequence (\ref{seq:decomp diff graded Q to P}) induces a long exact sequence in cohomology. Then, a direct computation shows:

\begin{proposition}
 \label{lem:boundary op Q to P}
 The boundary maps $\delta_P$ in the long exact sequence in cohomology associated to (\ref{seq:decomp diff graded Q to P}) are
 explicitly given by:
 \begin{equation*}
  \label{eq:boundary op Q to P}
  \begin{array}{cccc}
  \delta_P : & H^{0,k}(\ad P) & \rightarrow & H^{k+1}(\Omega^{\leqslant \bullet}) \\
             & [ \alpha ]               & \mapsto     &  [2c (\del^\theta \alpha \wedge F )].
  \end{array}
 \end{equation*}
Thus, a necessary condition for an infinitesimal deformation $[\alpha]\in H^{0,1}(\ad P)$ of $P$
to be lifted to an infinitesimal deformation of $Q$ is the vanishing of the class $\delta_P([\alpha]) \in H^{2}(\Omega^{\leqslant \bullet})$.
\end{proposition}

The obstruction given by the previous result can be regarded as the vanishing of the infinitesimal variation of the holomorphic Pontryagin class $p_1^c([P]) \in  H^{2}(\Omega^{\leqslant \bullet})$ in \eqref{eq:lescEind2cx}. Notice that $c(\del^\theta \alpha \wedge F) = \partial(c (\alpha F))$, and therefore on a $\partial\dbar$-manifold the map $\delta_P$ vanishes identically.

\subsection{Differential graded Lie algebra and Maurer-Cartan equation}


In this section we promote our resolution $(\cL^\bullet_Q, d_Q)$ of the sheaf of automorphisms $\faut_Q$ in Lemma \ref{lem:resolution sheaf Q} to a differential graded Lie algebra, or DGLA for short, in such a way that the corresponding Maurer-Cartan equation describes the integrable deformations of $(Q,P,\rho)$.

We start by recalling the definition of a DGLA.

\begin{definition}
\label{def:DGLA}
 A graded Lie algebra $(\fL^\bullet, [\cdot,\cdot])$ over $\C$ is a $\ZZ$-graded vector space $\fL=\bigoplus_{k\geq 0} \fL^k$ together with
 a family of bilinear maps $[\cdot,\cdot]: \fL^k\times \fL^l \to \fL^{k+l}$ satisfying, for any $(x,y,z)\in \fL^k\times \fL^l\times \fL^m$,
 \begin{enumerate}
  \item[i)] Graded skew-commutativity: $[x,y]+ (-1)^{kl} [y,x] = 0$.
  \item[ii)] Jacobi identity: $(-1)^{km}[x,[y,z]] + (-1)^{lk}[y,[z,x]]+(-1)^{ml}[z,[x,y]]=0$.
 \end{enumerate}
A differential graded Lie algebra $(\fL^\bullet,d,[\cdot,\cdot])$ is a graded Lie algebra $(\fL^\bullet,[\cdot,\cdot])$
together with a $\C$-linear map $d:\fL^\bullet \to \fL^{\bullet + 1}$, called differential, satisfying, for any $(x,y)\in \fL^k\times \fL^l$:
\begin{enumerate}
 \item[iii)] Differential: $d\circ d=0$.
 \item[iv)] Derivation of degree $1$: $d[x,y]=[dx,y]+ (-1)^{k}[x,dy]$.
\end{enumerate}
\end{definition}

To a DGLA, one associates a deformation equation, or Maurer-Cartan equation.

\begin{definition}
 \label{def:MC equation}
 Given a differential graded Lie algebra $(\fL^\bullet,d,[\cdot,\cdot])$, the Maurer-Cartan equation is
 \begin{equation*}
  \label{eq:MC equation}
  dx+\frac{1}{2}[x,x]=0, \: \textrm{ for } \: x\in \fL^1.
 \end{equation*}
\end{definition}

We will use the following results \cite{Kob,G2}:

\begin{enumerate}
\item The deformation theory for complex structure on $P$ is described by the DGLA $(\Om^{0,\bullet}(\ad P), \dbar^\theta,[\cdot,\cdot]_{P})$, with bracket extending the Lie bracket on $\fg$. In local coordinates, with a basis $(e_i)_{i=1..r}$ for $\fg$, given $g_1=g_1^i\otimes e_i\in\Om^{0,k}(\ad P)$ and
$g_2=g_2^i\otimes e_i\in\Om^{0,l}(\ad P)$, the bracket is:
\begin{equation*}
 \label{eq:bracket gs}
 [g_1,g_2]_{P}= \sum_{i,j= 1}^r g_1^i\wedge g_2^j \otimes [ e_i, e_j].
\end{equation*}
\item The deformation theory for holomorphic exact Courant algebroids is described by the abelian DGLA $(\Om^{\leqslant \bullet},d,0)$.
\end{enumerate}
In order to simplify the notation, from now on we will drop the subscripts on the brackets, the meaning being clear from the context.
To obtain a DGLA for the deformation theory of $Q$, we extend the Lie bracket on $\faut_Q(X)$
to $\fL_Q^\bullet$. An explicit formula for the bracket on $\faut_Q(X)$ follows from \eqref{eq:groupproduct}.

\begin{lemma}
The Lie bracket on $\faut_Q(X)$ is given by the formula
 \begin{equation*}
  \label{eq:bracket aut Q}
[(\alpha ,b), (\alpha ',b')]=([\alpha ,\alpha '], \: 2c(\del^\theta \alpha  \wedge \del^\theta \alpha ')).
 \end{equation*}
\end{lemma}

We extend this bracket to $\fL_Q^\bullet$ by setting, for $(\alpha ,b)\in \fL_Q^k$, $(\alpha ',b')\in \fL_Q^l$:
\begin{equation}
 \label{eq:bracket Q}
 [(\alpha ,b), (\alpha ',b')]:= ([\alpha , \alpha '], (-1)^k 2c(\del^\theta \alpha  \wedge \del^\theta \alpha '))\in \fL_Q^{k+l}.
\end{equation}


We will use the following lemma.

\begin{lemma}
 \label{lem:ad invariance c}
 Let $(x,y,z)\in \Om^{0,k}(\ad P) \times \Om^{0,l}(\ad P) \times \Om^{0,m}(\ad P)$. Then
 \begin{equation*}
  \label{eq:ad invariance c}
  c(x\wedge [y,z])+ (-1)^{kl}c([y,x] \wedge z)=0.
 \end{equation*}
\end{lemma}

\begin{proof}
This is a direct and local computation. Decomposing elements $x\in\Om^{0,k}(\ad P)$ in a basis $\{e_i\}$ of $\fg$
in the form $x= x_i\otimes e_i$, the result follows from the $\ad$-invariance of the pairing $c$.
\end{proof}

We are ready to prove the main result of this section, which describes the deformation theory for string algebroids in terms of a DGLA.


\begin{proposition}\label{theo:MC}
The expression (\ref{eq:bracket Q}) endows $(\fL^\bullet_Q,d_Q,[\cdot,\cdot])$ with the structure of a DGLA. Furthermore, 
given $(\alpha ,b)\in \fL^1_Q$, the Dolbeault operator $\theta^{0,1} + \alpha $ defines a new holomorphic $G$-bundle $P_{\theta+\alpha }$ endowed with a Courant extension determined by $(\theta + \alpha ,H - b)$ via Proposition \ref{prop:strholCour} if and only if $(\alpha ,b)$ satisfies the Maurer-Cartan equation
\begin{equation}\label{eq:MC}
d_Q(\alpha ,b)+\frac{1}{2}[(\alpha ,b),(\alpha ,b)]=0.
\end{equation}
\end{proposition}

\begin{proof}
Recall that $(\Om^{0,\bullet}(\ad P), \dbar^\theta,[\cdot,\cdot]_{P})$ is a DGLA, so it already satisfies the axioms of Definition \ref{def:DGLA}. From the expressions of $d_Q$ and $[\cdot,\cdot ]_Q$, we only need to consider the projection via $\pi_2$ in our computations. Consider $(x,y,z)\in \fL_Q^k\times\fL_Q^l\times\fL_Q^m$. We use the notation $\alpha _x = \pi_1 x$, $\alpha _y = \pi_1 y$, $\alpha _z = \pi_1 z$.
 We first show that $(\fL^\bullet_Q,[\cdot,\cdot])$ is a graded Lie algebra. The graded skew-commutativity is straightforward. The $\pi_2$-component of
 $$\mathrm{Jac}(x,y,z):=(-1)^{km}[x,[y,z]] + (-1)^{lk}[y,[z,x]]+(-1)^{ml}[z,[x,y]]$$
 is given by
 \begin{align*}
\pi_2 \mathrm{Jac}(x,y,z)  = {}& (-1)^{km+k} \: 2c(\del^\theta \alpha _x\wedge \del^\theta [\alpha _y,\alpha _z])  \\
                          & + (-1)^{lk+l}\:  2c(\del^\theta \alpha _y\wedge\del^\theta [\alpha _z,\alpha _x])  \\
                          &  + (-1)^{ml+m} \: 2c(\del^\theta \alpha _z\wedge \del^\theta [\alpha _x,\alpha _y]) \\
                        = {} & (-1)^{km+k} \: 2c(\del^\theta \alpha _x\wedge [\del^\theta \alpha _y,\alpha _z]) \\
                          & +(-1)^{lk+l+m} \: 2c(\del^\theta \alpha _y\wedge [\alpha _z\wedge\del^\theta \alpha _x]) + c.p.
 \end{align*}
where $c.p.$ stands for cyclic permutations. This vanishes by Lemma \ref{lem:ad invariance c}, so  the Jacobi identity holds. By Lemma \ref{lem:differential Q},  $d_Q$ is a differential. Finally, it only remains to verify the derivation property.
The $\pi_2$-projection of $d[x,y]$ in degree $k+l+1$ is
 \begin{align*}
 \frac{1}{2}\pi_2(d[x,y])  = {} & (-1)^k  d(c (\del^\theta \alpha _x\wedge \del^\theta \alpha _y) ) -  c(\del^\theta [\alpha _x,\alpha _y]\wedge F) \\
	        = {} & (-1)^k  c(\delb^\theta\del^\theta \alpha _x \wedge \del^\theta \alpha _y)-  (-1)^{k} c([\alpha _x,\del^\theta \alpha _y]\wedge F) \\
                  & -  c(\del^\theta \alpha _x \wedge \delb^\theta\del^\theta \alpha _y)-  c([\del^\theta \alpha _x, \alpha _y]\wedge F) \\
                = {} &    (-1)^{k+1}  c(\del^\theta\delb^\theta \alpha _x \wedge \del^\theta \alpha _y)+  
                           (-1)^k  c([F, \alpha _x] \wedge \del^\theta \alpha _y) \\  & - (-1)^{k} c([\alpha _x,\del^\theta \alpha _y]\wedge F)  + c(\del^\theta \alpha _x \wedge \del^\theta\delb^\theta \alpha _y)\\
                  &  -c(\del^\theta \alpha _x \wedge [F, \alpha _y])-  c([\del^\theta \alpha _x, \alpha _y]\wedge F) \\
                = {} &    (-1)^{k+1}  c(\del^\theta\delb^\theta \alpha _x \wedge \del^\theta \alpha _y)+ c(\del^\theta \alpha _x \wedge \del^\theta\delb^\theta \alpha _y)\\
                =  {} & \frac{1}{2}\pi_2([dx,y] + (-1)^k [x, dy] ),
\end{align*}
where we used Lemma \ref{lem:ad invariance c} in the fourth equality. This proves the first part of the statement.
%

As for the second part of the statement, the vanishing of the $\pi_1$-projection of \eqref{eq:MC} gives
$$
\delb^\theta \alpha  + \frac{1}{2} [\alpha ,\alpha ]=0,
$$
and therefore the operator $\theta^{0,1} + \alpha$ defines a new holomorphic $G$-bundle $P_{\theta+\alpha }$ on the underlying smooth $G$-bundle $\underline{P}$. The vanishing of the $\pi_2$-projection  is equivalent to
 \begin{eqnarray*}
           d(H - b) + c(F+\del^\theta \alpha\wedge F + \del^\theta \alpha) = 0.
\end{eqnarray*}
As $\alpha$ satisfies the Maurer-Cartan equation, the curvature of $\theta + \alpha$ is given by
 \begin{align*}
           F_{\theta+\alpha}  = {}& F_\theta + d^\theta \alpha + \frac{1}{2}[\alpha,\alpha] \\
                         = {}& F +\del^\theta \alpha.
\end{align*}
Thus $(H - b,\theta + \alpha)$ defines a Courant extension of the holomorphic principal bundle $P_{\theta + \alpha}$, via Proposition \ref{prop:strholCour}. 
\end{proof}

\begin{remark}
Observe  that in the sequence (\ref{seq:decomp diff graded Q to P}) the brackets commute with the inclusion and projection, and we obtain an exact sequence of DGLAs.
\end{remark}

\subsection{Kuranishi slice}

Let $(Q,P,\rho)$ be a string algebroid over a compact complex manifold $X$. In this section we construct a locally complete family of deformations of $(Q,P,\rho)$ following the method of Kuranishi \cite{ku}. The definition of the \emph{gauge group} of symmetries is a subtle question, which involves `a-field' transformations (see \cite{grt,Rubioth}).

We fix a smooth principal $G$-bundle $\underline{P}$ over $X$ with vanishing first Pontryagin class. Consider the space of parameters 
$$
\cP = \Omega^{\leqslant 1} \times \cA_{\underline{P}},
$$ 
where $\cA_{\underline{P}}$ denotes the space of connections on $\underline{P}$. We have the subspace of integrable pairs
$$
\cP^I = \{(H,\theta) \; | \; F_\theta^{0,2} = 0, \; dH + c(F_\theta \wedge F_\theta) = 0\} \subset \cP.
$$
By Proposition \ref{prop:strholCour}, a point $(H,\theta) \in \cP^I$ determines a Courant extension $(Q,P_\theta,\rho)$ of the holomorphic principal $G$-bundle $P_\theta = (\underline{P},\theta^{0,1})$. Note that we can identify the tangent space of $\cP$ at $(H,\theta)$ with
$$
T_{(H,\theta)} \cP = \Omega^{\leqslant 1} \oplus \Omega^1(\ad \underline{P}),
$$ 
and the subspace $T_{(H,\theta)}\cP^I \subset T_{(H,\theta)}\cP$ corresponds to pairs $(\dot H,\dot \theta)$ solving  
\begin{equation}\label{eq:integrableinf}
\dbar^\theta \dot \theta^{0,1} = 0, \qquad d(\dot H + 2c(\dot \theta \wedge F_\theta)) = 0.
\end{equation}

Define the \emph{gauge group} of symmetries by
$$
\cG := \cG_{\underline{P}} \times \Omega^{1,0}(\ad \underline{P}) \times \Omega^{2,0},
$$
where $\cG_{\underline{P}}$ denotes the gauge group of $\underline{P}$, with group operation
$$ 
(g,a,B) (g',a',B') = (gg', B + B' + c((g'^{-1} \cdot a)\wedge a'), g'^{-1} \cdot a + a').
$$
The associativity for the product follows from \cite[Prop. 4.3]{grt}. Note that the vector space underlying the Lie algebra of $\cG$ is given by
$$
\Lie \; \cG = \Omega^0(\ad \underline{P}) \oplus \Omega^{1,0}(\ad \underline{P}) \oplus \Omega^{2,0}.
$$
The proof of the following lemma is straightforward.

\begin{lemma}\label{lem:gaugeaction}
The group $\cG$ acts on $\cP$ preserving $\cP^I \subset \cP$, by the formula
$$
(g,a,B) \cdot (H,\theta) = (H + dB + CS(\theta) - CS(\theta + a) - d c (\theta \wedge a),g (\theta + a)).
$$
The infinitesimal action $L \colon \Lie \; \cG \to T_{(H,\theta)}\cP$ 
at $(H,\theta)$ is given by
$$
L(\alpha,a,b) = (db - 2c(a \wedge F_\theta),-d^\theta \alpha + a).
$$

\end{lemma}

Our next result shows that the equivalence relation induced by gauge transformations is equivalent to the one induced by isomorphisms of string algebroids in Definition \ref{def:stringholCourmor}. Given a string algebroid $(Q,P,\rho)$, we consider its group of automorphisms $\Aut(Q,P,\rho)$  (see Lemma \ref{lemma:Autstring}).

\begin{lemma}\label{lem:gaugerel}
Let $(H,\theta), (H',\theta') \in \cP^I$ and let $(Q,P,\rho)$ and $(Q',P',\rho')$ be the associated string algebroids, respectively. Then, $(H',\theta')$ is in the $\cG$-orbit of $(H,\theta)$ if and only if $(Q',P',\rho')$ is isomorphic to $(Q,P,\rho)$. Furthermore, there is a natural identification between the isotropy group of $(H,\theta)$ in $\cG$ and $\Aut(Q,P,\rho)$.
\end{lemma}
\begin{proof}
If $(Q',P',\rho')$ is isomorphic to $(Q,P,\rho)$, then there exists $g' \in \cG_{\underline{P}}$ such that $g'\theta^{0,1} = \theta'^{0,1}$. Then $g'^{-1} \theta' \in \cA_{P}$, and by Lemma \ref{lemma:strisomorphism} there exists $B \in \Omega^{2,0}$ and $g \in \cG_P$ satisfying
\begin{align*}
H'  & = H + CS(g\theta) - CS(g'^{-1}\theta') - dc(g\theta \wedge g'^{-1}\theta') + dB\\
& = H + CS(\theta) - CS((g'g)^{-1}\theta') - dc(\theta \wedge (g'g)^{-1}\theta') + dB.
\end{align*}
Setting $a = (g'g)^{-1}\theta' - \theta$ we obtain 
$$
(g'g,a,B) \cdot (H,\theta) = (H',\theta').
$$ 
Conversely, if $(H',\theta')$ is in the $\cG$-orbit of $(H,\theta)$ we can assume that $P = P'$, that is $\theta'^{0,1} = \theta^{0,1}$. Then, $(g,a,B) \cdot (H',\theta') = (H,\theta)$ implies that $g$ is an automorphism of $P$ and therefore $g\theta' \in \cA_{P}$. From Lemma \ref{lem:gaugeaction} and the functoriality of the Chern-Simons $3$-form (see \eqref{eq:CSfunctorial}) it follows that 
$$
H = H' + dB + CS(g \theta') - CS(\theta) - d c (g \theta' \wedge \theta),
$$
and therefore $(Q',P',\rho')$ is isomorphic to $(Q,P,\rho)$ by Proposition \ref{prop:classification}. Finally, if  $(g,a,B) \cdot (H,\theta) = (H,\theta)$ then $a = g^{-1} \theta - \theta$ and
$$
dB = CS(g^{-1} \theta) - CS(\theta) - d c (g^{-1} \theta \wedge \theta).
$$
Thus, the map $(g,a,B) \to (g,B)$ induces an isomorphism between the isotropy group of $(H,\theta)$ in $\cG$ and $\Aut(Q,P,\rho)$.
\end{proof}

\begin{remark}
A priori, one may have thought that the presence of $a$-field transformations in the definition of $\cG$ implies that the equivalence relation by gauge transformations is stronger than the one induced by Definition \ref{def:stringholCourmor}. Key to the proofs of Lemma \ref{lem:gaugeaction} and Lemma \ref{lem:gaugerel} is the fact that our a-field transformations are of type $(1,0)$. These symmetries balance the redundant degrees of freedom in the parameters introduced by a choice of connection in the proof of Theorem \ref{th:classification}.
\end{remark}

We next provide the relation with the DGLA in Theorem \ref{theo:MC}. 
Note that a simple calculation shows that the complex $(\fL_Q^\bullet,d_Q)$ in Lemma \ref{lem:resolution sheaf Q} is elliptic.
Note also that $\fL_Q^1$ carries a linear complex-structure given by multiplication of each factor by $i$, which descends to $H^1(\fL_Q^\bullet,d_Q)$.
Lastly, $\cP$ carries a formal almost complex structure $J_\cP$ given by
$$
J_\cP(\dot H, \dot \theta) = (i \dot H, -\dot \theta( J \cdot)),
$$
for $(\dot H,\dot \theta)\in T_{H,\theta}\cP$, and where $J$ denotes the almost complex structure of $X$.
\begin{lemma}
\label{lem:elliptic}
Let $(H,\theta) \in \cP^I$ and $(Q,P,\rho)$ be the associated string algebroid. The holomorphic map
\begin{equation}\label{eq:hatepsilonmap}
\hat \epsilon \colon \fL_Q^1 \to \cP,
\end{equation}
defined by $\hat \epsilon(\alpha,b) = (H- b,\theta + \alpha)$, intersects every $\cG$-orbit in $\cP$. Furthermore, its differential at the origin induces an isomorphism
\begin{equation}\label{eq:dhatepsilonmap}
d \hat \epsilon_0 \colon H^1(\fL_Q^\bullet,d_Q) \to T_{(H,\theta)}\cP^I/\operatorname{Im} L.
\end{equation}
\end{lemma}

\begin{proof}
Given $(H',\theta') \in \cP$, we define $a = (\theta - \theta')^{1,0}$, $\alpha = (\theta' - \theta)^{0,1}$ and 
$$
b = H - H' - (CS(\theta) - CS(\theta + a) - d c (\theta \wedge a)),
$$
Then, the first part of the statement follows from
\begin{align*}
(0,a,0)(H',\theta') & = (H' + (H-H'-b),\theta^{1,0} + \theta'^{0,1}) = \hat \epsilon (\alpha,b).
\end{align*}
As for the second part, notice that the map 
$$
\varphi \colon T_{(H,\theta)}\cP \to \fL_Q^1
$$
defined by $\varphi(\dot H,\dot \theta) = (\dot\theta^{0,1},- \dot H - 2c(\dot \theta^{1,0} \wedge F_\theta))$ satisfies $\varphi \circ d \hat \epsilon_0 = \Id$, and therefore it is enough to prove that $\varphi$ induces a well-defined isomorphism
\begin{equation}\label{eq:varphimap}
\varphi \colon T_{(H,\theta)}\cP^I/\operatorname{Im} L \to H^1(\fL_Q^\bullet,d_Q).
\end{equation}
To prove this, given $(\dot H,\dot \theta) \in T_{(H,\theta)}\cP^I$ (see \eqref{eq:integrableinf}) we have
$$
- d(\dot H + 2c(\dot \theta^{1,0} \wedge F_\theta)) - 2c(\partial^\theta \dot \theta^{0,1} \wedge F_\theta) = - d(\dot H + 2c(\dot \theta \wedge F_\theta)) = 0,
$$
and therefore $d_Q(\varphi(\dot H,\dot \theta)) = 0$. Furthermore, if $v = L(\alpha,a,b)$ 
we have
$$
\varphi(v) = (- \dbar^\theta \alpha,  - db + 2c(\partial^\theta \alpha \wedge F_\theta)) \in \operatorname{Im} d_Q,
$$
and therefore \eqref{eq:varphimap} is well-defined. To see the injectivity, if $(\dot H,\dot \theta) \in T_{(H,\theta)}\cP^I$ satisfies
$$
\varphi(\dot H,\dot \theta) = (\dbar^\theta \alpha, db - 2c(\partial^\theta \alpha \wedge F_\theta))
$$
for $(\alpha,b) \in \fL_Q^0$, then $\dot \theta = d^\theta \alpha + a$ and 
$$
\dot H = - db + 2c(\partial^\theta \alpha \wedge F_\theta)) - 2c(\dot \theta^{1,0} \wedge F_\theta)) = - db - 2c(a \wedge F_\theta),
$$ 
for $a = \dot\theta^{1,0}- \partial^\theta \alpha$. Thus, $(\dot H,\dot \theta) \in \operatorname{Im} L$ and $\varphi$ is injective. For the surjectivity, if $(\alpha,b)$ is a $d_Q$-closed, then we define $(\dot H,\dot\theta) = (-b,\alpha)$, which clearly defines an element in $T_{(H,\theta)}\cP^I$ such that $\varphi(\dot H,\dot\theta) = (\alpha,b)$.
\end{proof}

We are now ready to prove the main result of this section.

\begin{theorem}
\label{thm:slice}
There exists a map $\epsilon$ from a neighbourhood $U$ of $0\in H^1(\fL_Q^\bullet,d_Q)$ to a neighbourhood of $(H,\theta)\in \cP$ such that
\begin{enumerate}
 \item The map $\epsilon$ is 
 holomorphic, and $\epsilon(0)=(H,\theta)$,
 \item each $\cG$-orbit of integrable pairs near $(H,\theta)$ intersects the image of $\epsilon$,
 \item there is an analytic obstruction map $\Phi \colon H^1(\fL_Q^\bullet,d_Q) \to H^2(\fL_Q^\bullet,d_Q)$ with $\Phi(0) = 0$ and $d\Phi(0) = 0$, 
 such that the deformations parametrized by the sub-family $\cM = \{x \in H^1(\fL_Q^\bullet,d_Q) \; | \; \Phi(x) = 0\}$ are precisely the integrable ones. 
\end{enumerate}
\end{theorem}
In the case that the obstruction map vanishes, it follows from the proof of Theorem \ref{thm:slice} that
$\lbrace \epsilon (x),\: x\in \cM \rbrace$ defines an analytic locally complete family of deformations of $(Q,P,\rho)$.
More precisely, any analytic string algebroid deformation $(Q_t,P_t,\rho_t)_{t\in\Lambda}$ of $(Q,P,\rho)$
is induced by pullback of the family $\lbrace \epsilon (x),\: x\in \cM \rbrace$ via an analytic map $\Lambda \to \cM$.
\begin{proof}
We choose hermitian metrics on $X$ and $\ad \underline{P}$, which enable to define $L^2$-pairings and formal adjoints $d_Q^*$ in the complex $(\fL_Q^\bullet,d_Q)$. We set $$\Delta= d_Q d_Q^* + d_Q^*d_Q$$ and define
$$
\cH^1:= ker(\Delta)=\lbrace x\in \fL_Q^1 \;:\; d_Q x=0 \text{ and } d_Q^*x=0 \rbrace \cong H^1(\fL_Q^\bullet,d_Q).
$$
Define an operator $E=(E_1,E_2):\fL_Q^1\to \fL_Q^2$ by
$$
E = d_Q + \frac{1}{2}[\cdot,\cdot].
$$
Set
$$
\Psi=\lbrace \psi\in  \fL_Q^1 \;:\;  E(\psi) = 0,\; d_Q^*(\psi)=0\rbrace.
$$
Consider the operator $$D= d_Q - E = - \frac{1}{2}[\cdot,\cdot]$$ and the Green operator $G$ of $\Delta$ on $\fL_Q^1$. For any $\psi\in \Psi$, we have
$$
\Delta \psi -d_Q^*D( \psi)=0.
$$
By applying the Green operator to this equality, we obtain, for any $\psi\in \Psi$, 
\begin{equation*}
\label{eq:harmonicpsi}
\psi - G d_Q^* D(\psi)\in \cH^1.
\end{equation*}
Define
$$
\begin{array}{cccc}
 F_1 : & \fL_Q^1 & \rightarrow & \fL_Q^1 \\
 & \psi & \mapsto & \psi - Gd_Q^* D(\psi).
\end{array}
$$
Using the metrics on $X$ and $\ad \underline{P}$ together with the covariant derivative $d^\theta$, we endow the spaces $\fL_Q^*$ with Sobolev norms $\vert\vert\cdot\vert\vert_k$. 
The operator $D$ is a composition of linear and bilinear continuous operators with respect to these norms, so $F_1$ extends in a unique way to a continuous operator from the Banach completion $\cB^k$ of $\fL_Q^1$ with respect to $\vert\vert\cdot\vert\vert_k$ to itself (here we assume that $k$ is large enough, so that $\cB^k$ is a Banach algebra). We denote also this extension by $F_1$. Since its differential at $0$ is the identity, we can define, by the inverse function theorem, a continuous local inverse of $F_1$ at $0$ from a neighbourhood $B$ of $0$ in $\cB^k$ to itself.

Let $U=B\cap \cH^1$. For each $x \in U$ 
we have that $\Delta F_1^{-1}(x) - d_Q^* D F_1^{-1}(x)$ is harmonic, and by elliptic regularity $F_1^{-1}(x)$ is smooth. 
For any $x \in U$, given by $x = F_1(\alpha_x,b_x)$ for some $(\alpha_x,b_x)$, we set
$$
\epsilon(x) = \hat \epsilon(b_x,\alpha_x) = (H-b_x,\theta + \alpha_x),
$$
where $\hat \epsilon$ is as in Lemma \ref{lem:elliptic}, which defines a map
$$
\epsilon \colon U \subset \cH^1 \to \cP
$$
such that $\epsilon(0)=(H,\theta)$. We prove now that $\epsilon$ has the desired properties. As $F_1$ is a composition of linear and bilinear maps, and the map $\hat \epsilon$ 
is holomorphic, so is $\epsilon$ and $\epsilon$ satisfies $(1)$. 

To prove $(2)$, we define a map $\Upsilon:\cP \to \cL_Q^1$ by $\Upsilon(H', \theta')= ( (\theta ' - \theta )^{0,1}, H- H')$.
Let $(\operatorname{Ker} d_Q)^{\perp}$ be the orthogonal of $\operatorname{Ker} d_Q$ in $\fL_Q^0$. We define
$$
\begin{array}{cccc}
F_2 : & (\operatorname{Ker} d_Q)^\perp \times \fL_Q^1 & \rightarrow & \fL_Q^0
\end{array}
$$
by the formula 
$$
F_2(\alpha_0,b_0,\alpha_1,b_1) = d_Q^*\Upsilon((\exp(\alpha_0),\theta^{1,0} - \exp(-\alpha_0)\theta^{1,0},b_0)(H-b_1,\theta + \alpha_1)),
$$
where $(\exp(\alpha_0),\theta^{1,0} - \exp(-\alpha_0)\theta^{1,0},b_0)$ is regarded as an element of $\cG$ acting on $(H-b_1,\theta + \alpha_1) \in \cP$. 
The differential of $F_2$ with respect to $\alpha_0,b_0$ at $0$ is $d_Q^*d_Q$, and we proceed just as for the map $F_1$: we use the implicit function theorem and obtain a map $\psi \mapsto (\alpha_0(\psi),b_0(\psi))$ from a neighbourhood of $0\in \fL_Q^1$
to a neighbourhood of zero in $(\operatorname{Ker} d_Q)^{\perp}$ satisfying $F_2(\alpha_0(\psi),b_0(\psi),\psi) = 0$. By Lemma \ref{lem:elliptic}, every deformation of $(H,\theta)$ is gauge equivalent to some $\hat \epsilon (\psi)$ in the image of $\hat \epsilon$. Furthermore, if $\psi$ is sufficiently small and satisfies the Maurer-Cartan equation, then $\hat \epsilon (\psi)$ is gauge equivalent to a small deformation lying in $\Psi$, which corresponds to an element in our finite-dimensional family.

Finally, to prove (3), we understand the condition for $\epsilon(x)$ to be an integrable pair. 
By construction, this is equivalent to $E(\alpha_x,b_x) = 0$, and we have
\begin{align*}
E(\alpha_x,b_x) & = (d_Q - D)(\alpha_x,b_x)\\
& = d_Q x + (d_Q G d_Q^*D - D) (\alpha_x,b_x)\\
& = - (\Pi + d_Q^*Gd_Q)D(\alpha_x,b_x),
\end{align*}
where $\Pi$ denotes the orthogonal projection
$$
\Pi \colon \fL_Q^2 \to \cH^2:= \lbrace x\in \fL_Q^2 \;:\; d_Q x=0 \text{ and } d_Q^*x=0 \rbrace \cong H^2(\fL_Q^\bullet,d_Q).
$$
Since the images of $\Pi$ and $d_Q^*$ are orthogonal, $E(\alpha_x,b_x) = 0$ if and only if 
$\Pi D(\alpha_x,b_x) = d_Q^*Gd_Q D(\alpha_x,b_x) = 0$. By Proposition \ref{theo:MC}  $(\fL_Q^\bullet,d_Q,[\cdot,\cdot])$ is a DGLA, and we can use the argument of
Kuranishi \cite{ku} which, using the compatibility between $[\cdot,\cdot]$ and $d_Q$, shows that $d_Q^*Gd_Q D(\alpha_x,b_x)$ vanishes if $\Pi D(\alpha_x,b_x)$ does. Hence, $\epsilon(x)$ is integrable precisely when $x$ lies in the vanishing set of the analytic mapping $\Phi \colon H^1(\fL_Q^\bullet,d_Q) \to H^2(\fL_Q^\bullet,d_Q)$ defined by
$$
\Phi(x) = \Pi D(\alpha_x,b_x).
$$
The conditions $\Phi(0) = 0$ and $d\Phi(0) = 0$ are satisfied by construction. Note that, when $\Phi = 0$, the locally complete family of deformations $\cM$ is simply given by the open set $U$, and hence is smooth.
\end{proof}


\appendix

\section{String algebroids in the smooth category}\label{sec:smooth}

For completeness, in this section we give a brief account of the classification of string algebroids in the smooth category. The proofs of Section \ref{sec:holomorphiccase} can be easily adapted to the present situation.  
Our analysis shall be compared with the one made in \cite{Bressler,ChStXu} for regular and transitive Courant algebroids.

Let $G$ be a real Lie group whose Lie algebra $\mathfrak{g}$ admits a bi-invariant pairing $c \colon \mathfrak{g} \otimes \mathfrak{g} \to \RR.$ Let $M$ be a smooth manifold. Smooth principal $G$-bundles over $M$ will be denoted by $P$. Their gauge group will be denoted $\cG_P$.
The definition of smooth string algebroid $(Q,P,\rho)$ with structure group $G$ and underlying pairing $c$ and the notion of morphism are completely analogue to Definition \ref{def:stringholCour} and Definition \ref{def:stringholCourmor}, replacing `holomorphic' by `smooth', and therefore we will omit them.




As for the classification, consider $\mathcal{C}_G$ and $\Omega^2$, the sheaves of smooth $G$-valued functions on $M$ and smooth two-forms on $M$, respectively, and define the subsheaf
$$ \cS \subset \mathcal{C}_G\times \Omega^2 $$
by the analogue to \eqref{eq:dBCShol}.
As in the holomorphic case, Theorem \ref{th:classification}, we have the following:

\begin{theorem}\label{prop:classificationsmooth}
	There is a one to one correspondence between elements in $H^1(\cS)$ and isomorphism classes of smooth string algebroids. 
\end{theorem}

We then describe the structure of $H^1(\cS)$.
The short exact sequence of sheaves of groups \begin{equation*}\label{eq:sescEsmooth}
0 \to \Omega^2_{cl} \to \cS \to \cC_G \to 1.
\end{equation*}
induces a long exact sequence (of pointed sets) in cohomology
\begin{equation}\label{eq:lescEind}
\xymatrix@R-2pc{
	0 \ar[r] & H^0(\Omega^2_{cl}) \ar[r] & H^0(\cS) \ar[r] & H^0(\cC_G) \ar[r]^{\quad \delta_1} & \\
	\ar[r] & H^1(\Omega^2_{cl}) \ar[r] & H^1(\cS) \ar[r] & H^1(\cC_G) \ar[r]^{\delta_2} & H^2(\Omega^2_{cl}),
}
\end{equation}
where $H^1(\cC_G)$ parameterizes isomorphism classes of principal $G$-bundles on $M$. Recall the canonical isomorphisms of abelian groups $H^i(\Omega^2_{cl}) \cong H^{2+i}(M,\RR)$ for any $i \geq 1$. 

To clarify the meaning of $\delta_1$ in the previous sequence, we use the following analogue of Lemma \ref{lemma:sigmaP}

\begin{lemma}\label{lemma:sigmaPbis}
Let $P$ be a smooth principal $G$-bundle over $M$. Then, there is an homomorphism of groups
\begin{equation}\label{eq:sigmaPbis}
\sigma_P \colon \cG_P \to H^3(M,\RR)
\end{equation}
defined by 
$$
\sigma_P(g) = [CS(g \theta) - CS(\theta) - d c (g \theta \wedge \theta)] \in H^3(M,\RR),
$$
for any choice of connection $\theta$ on $P$.
\end{lemma}

We define now a subgroup $I$ of the additive group $H^3(M,\RR)$, induced by $\delta_1$ in \eqref{eq:lescEind}.  The following result is analogue to Lemma \ref{lemma:delta1}.


\begin{lemma}\label{lemma:delta1bis}
Via the isomorphism $H^1(\Omega^2_{cl}) \cong H^3(M,\RR)$, the map $\delta_1$ in \eqref{eq:lescEindcx} is given by
$$
\delta_1(g) = \sigma_{M \times G}(g) = [g^*\sigma_3],
$$ 
for $g \in C^\infty(M,G)$ and $\sigma_3$ the Cartan $3$-form on $G$ associated to $c$. Consequently, $\delta_1$ is a morphism of groups, and we define
$$
I := \operatorname{Im} \delta_1 \subset H^3(M,\RR).
$$
\end{lemma}

For the next result, recall that 
any isomorphism class of smooth principal $G$-bundles $[P] \in H^1(\cC_G)$ defines a first Pontryagin class
$$
p_1([P]) \in H^4(M,\RR),
$$
represented by $c(F_\theta \wedge F_\theta) \in \Omega^4$ for any choice of connection $\theta$ on $P$. 
The next result is analogue to Proposition \ref{prop:lescEindcx}.

\begin{proposition}\label{prop:lescEind}
	There is an exact sequence of pointed sets
	\begin{equation}\label{eq:lescEind2}
	\xymatrix{
		0 \ar[r] & H^3(M,\RR)/I \ar[r]^{\quad \iota} & H^1(\cS) \ar[r]^\jmath & H^1(\cC_G) \ar[r]^{^{p_1}\quad } & H^4(M,\RR).
	}
	\end{equation}
	Furthermore,  $\iota$ induces a transitive action of the additive group $H^3(M,\RR)$ on the fibres of the map $H^1(\cS)\xrightarrow{j} H^1(\cC_G)$.
\end{proposition}



\begin{remark}
	When $G$ is abelian, $\sigma_3 = 0$ and thus $H^1(\cS)$ is a $H^3(M,\RR)$-torsor, so Proposition \ref{prop:lescEind} recovers \cite[Prop. 2.17]{Rubioth}.
\end{remark}

The description of the fibre $\jmath^{-1}([P])$ in \eqref{eq:lescEind2} \emph{\`a la} de Rham is as follows. In the statement, $\cG_P$ denotes the group of gauge transformations of $P$.


\begin{proposition}\label{theo:deRham}
	Let $P$ be a principal $G$-bundle over $M$. Denote by $\mathcal{A}_P$ the space of connections on $P$. Then, there is a natural bijection
	\begin{equation*}\label{eq:lescEind3smooth}
	\jmath^{-1}([P]) \cong \{(H,\theta) \in \Omega^3 \times \cA_P \; | \; dH = c(F_\theta \wedge F_\theta)\}/ \sim,
	\end{equation*}
	where $(H,\theta)\sim (H',\theta')$ if, for some $B\in\Omega^2$, and $g\in \cG_P$ 
	\begin{equation}\label{eq:anomalysmooth}
	H'  = H + CS(g\theta) - CS(\theta') - dc(g\theta \wedge \theta') + dB.
	\end{equation}
\end{proposition}

Using Proposition \ref{theo:deRham}, we finish this section establishing the relation between smooth string algebroids and the notion of real string class on a principal $G$-bundle, as introduced by Redden \cite{Redden}. Recall that integral string classes parametrize string structures on $P$ up to homotopy. When $G = \Spin(n)$, string structures are in correspondence with lifts of the classifying map $M \to BG$ determined by $P$, to the classifying space of the \emph{string group}. This provides further motivation for Definitions \ref{def:stringholCour} and \ref{def:stringholCourmor}.

\begin{definition}[\cite{Redden}]\label{def:stringclass}
	A \emph{real string class} on $P$ is a class $[\hat H] \in H^3(P,\RR)$ such that the restriction of $[\hat H]$ to any fibre of $P$ coincides with the class of the Cartan three-form $[\sigma_3] \in H^3(G,\RR)$.
\end{definition}

String classes form a torsor over $H^3(M,\RR)$, that we shall denote $H^3_{str}(P,\RR)$, where the action is defined by pullback and addition \cite[Prop. 2.16]{Redden}
$$
[\hat H] \to [\hat H] + p^*[H],
$$
where $p \colon P \to M$ is the canonical projection on the principal bundle $P$ and $[H] \in H^3(M,\RR)$. 
To describe $\jmath^{-1}([P])$ in terms of real string classes, we consider the additive subgroup
	$$
	I_{[P]} := \Im \; \sigma_P \subset H^3(M,\RR).
	$$
Note that $I = I_{M \times G}$ in Lemma \ref{lemma:delta1bis}.

\begin{proposition}\label{lemma:deRhamI}
Using the characterization of $\jmath^{-1}([P])$ in Theorem \ref{theo:deRham}, there is an isomorphism of $H^3(M,\RR)/I_{[P]}$-torsors
\begin{equation*}\label{eq:torsoriso}
\jmath^{-1}([P]) \cong H^3_{str}(P,\RR)/I_{[P]}
\end{equation*}
defined by $[(H,\theta)]  \mapsto [p^*H - CS(\theta)]$
\end{proposition}

\begin{proof}
	Note first that \eqref{eq:anomalysmooth} implies that $[(H,\theta)]  \mapsto [p^*H - CS(\theta)]$ is well defined, and it is equivariant with respect to the natural action of $H^3(M,\RR)$ on the domain and the target. The proof follows from the fact that $H^3_{str}(P,\RR)/I_{[P]}$ is a $H^3(M,\RR)/I_{[P]}$-torsor, by \cite[Prop. 2.16]{Redden}.
\end{proof}

As observed in \cite{BarHek,GF}, an equivariant exact Courant algebroid $\hat E$ over the total space of $P$ equipped with a (trivially) extended action corresponds to a cohomology class $[H - p^* CS(\theta)] \in H^3(P,\RR)$. The reduction of $\hat E$ defines a string algebroid on $M$ in the class $[(H,\theta)] \in \jmath^{-1}([P])$. In fact, the torsor of extended actions modulo isomorphism can be identified with $H^3_{str}(P,\RR)$ (see \cite[Prop. 3.7]{BarHek}). Proposition \ref{lemma:deRhamI} answers a question by Hekmati, about the discrepancy between the isomorphism classes of the reduced Courant algebroids and the torsor $H^3_{str}(P,\RR)$.

%
%


\end{document}